\documentclass[draftcls, onecolumn]{IEEEtran}
\usepackage{cite,color,graphicx,url,amssymb,amsthm,threeparttable,multirow,algorithm,algorithmic}
\usepackage[tight,footnotesize]{subfigure}
\usepackage[cmex10]{amsmath}
\IEEEoverridecommandlockouts

\newcommand{\bzero}{\mathbf{0}}
\newcommand{\bone}{\mathbf{1}}

\newcommand{\cA}{\mathcal{A}}

\newcommand{\C}{\mathbb{C}}

\newcommand{\cC}{\mathcal{C}}

\newcommand{\E}{\mathbb{E}}
\newcommand{\cE}{\mathcal{E}}
\newcommand{\cF}{\mathcal{F}}

\newcommand{\cG}{\mathcal{G}}

\newcommand{\cI}{\mathcal{I}}

\newcommand{\cN}{\mathcal{N}}
\newcommand{\cCN}{\mathcal{CN}}

\newcommand{\bbP}{\mathbb{P}}

\newcommand{\R}{\mathbb{R}}

\newcommand{\cS}{\mathcal{S}}
\newcommand{\whcS}{\widehat{\cS}}

\newcommand{\cT}{\mathcal{T}}

\newcommand{\teta}{\tilde{\eta}}

\newcommand{\defn}{\stackrel{def}{=}}
\newcommand{\nN}[1]{\left[\!\left[{#1}\right]\!\right]}

\DeclareMathOperator*{\argmin}{arg\,min}

\newcommand{\tH}{\mathrm{H}}
\newcommand{\tT}{\mathrm{T}}

\newcommand{\SNR}{\textsf{\textsc{snr}}}
\newcommand{\MAR}{\textsf{\textsc{mar}}}
\newcommand{\LAR}{\textsf{\textsc{lar}}}

\theoremstyle{plain}
\newtheorem{lemma}{Lemma}
\newtheorem{theorem}{Theorem}

\newtheorem{proposition}{Proposition}

\theoremstyle{definition}
\newtheorem{definition}{Definition}
\theoremstyle{remark}
\newtheorem{remark}{Remark}

\begin{document}

\title{Why Gabor Frames? Two Fundamental Measures of Coherence and Their\\
	Role in Model Selection}

\author{Waheed U. Bajwa, Robert Calderbank, and Sina Jafarpour%
%
\vspace{-0.8em}
%
\thanks{This paper was presented in part at the IEEE International Symposium on Information Theory, Austin, TX, June 2010. WUB is with the Program in Applied and Computational Mathematics, RC is with the Department of Electrical Engineering and the Program in Applied and Computational Mathematics, and SJ is with the Department of Computer Science at Princeton University, Princeton, NJ 08544 (Emails: {\tt wbajwa@math.princeton.edu, calderbk@math.princeton.edu, sina@cs.princeton.edu}).}
}


\maketitle

\begin{abstract}
The problem of model selection arises in a number of contexts, such as subset selection in linear regression, estimation of structures in graphical models, and signal denoising. This paper studies non-asymptotic model selection for the general case of arbitrary (random or deterministic) design matrices and arbitrary nonzero entries of the signal. In this regard, it generalizes the notion of \emph{incoherence} in the existing literature on model selection and introduces two fundamental measures of coherence---termed as the worst-case coherence and the average coherence---among the columns of a design matrix. It utilizes these two measures of coherence to provide an in-depth analysis of a simple, model-order agnostic one-step thresholding (OST) algorithm for model selection and proves that OST is feasible for exact as well as partial model selection as long as the design matrix obeys an easily verifiable property, which is termed as the \emph{coherence property}. One of the key insights offered by the ensuing analysis in this regard is that OST can successfully carry out model selection even when methods based on convex optimization such as the lasso fail due to the rank deficiency of the submatrices of the design matrix. In addition, the paper establishes that if the design matrix has reasonably small worst-case and average coherence then OST performs near-optimally when either (i) the energy of any nonzero entry of the signal is close to the average signal energy per nonzero entry or (ii) the signal-to-noise ratio in the measurement system is not too high. Finally, two other key contributions of the paper are that (i) it provides bounds on the average coherence of Gaussian matrices and Gabor frames, and (ii) it extends the results on model selection using OST to low-complexity, model-order agnostic recovery of sparse signals with arbitrary nonzero entries. In particular, this part of the analysis in the paper implies that an Alltop Gabor frame together with OST can successfully carry out model selection and recovery of sparse signals irrespective of the phases of the nonzero entries even if the number of nonzero entries scales almost linearly with the number of rows of the Alltop Gabor frame.
\end{abstract}


\section{Introduction}
\subsection{Background}
In many information processing and statistics problems involving high-dimensional data, the ``curse of dimensionality'' can often be broken by exploiting the fact that real-world data tend to live in low-dimensional manifolds. This phenomenon is exemplified by the important special case in which a data vector $\beta \in \C^p$ satisfies $\|\beta\|_0 \doteq \sum_{i=1}^p 1_{\{|\beta_i| > 0\}} \leq k \ll p$ and is observed according to the linear measurement model $y = X \beta + \eta$. Here, $X$ is an $n \times p$ (real- or complex-valued) matrix called the \emph{measurement} or \emph{design matrix}, while $\eta \in \C^n$ represents noise in the measurement system. In this problem, the assumption that the data vector $\beta$ is ``$k$-sparse'' allows one to operate in the so-called ``compressed'' setting, $k < n \ll p$, thereby enabling tasks that might be deemed prohibitive otherwise because of either technological or computational constraints.

Fundamentally, given a measurement vector $y = X\beta + \eta$ in the compressed setting, there are three complementary---but nonetheless distinct---questions that one needs to answer:
\begin{itemize}
 \item[] \hspace{-1em}\textbf{[Estimation]} Under what conditions can we obtain a reliable estimate of a $k$-sparse $\beta$ from $y$?
 \item[] \hspace{-1em}\textbf{[Regression]} Under what conditions can we reliably approximate $X\beta$ corresponding to a $k$-sparse $\beta$ from $y$?
 \item[] \hspace{-1em}\textbf{[Model Selection]} Under what conditions can we reliably recover the locations of the nonzero entries of a $k$-sparse $\beta$ (in other words, the model $\cS \doteq \{i \in \{1,\dots,p\}:|\beta_i| > 0\}$) from $y$?
\end{itemize}
A number of researchers have attempted to address the estimation and the regression question over the past several years. In many application areas, however, the model-selection question is equally---if not more---important than the other two questions. In particular, the problem of model selection (sometimes also known as \emph{variable selection} or \emph{sparsity pattern recovery}) arises indirectly in a number of contexts, such as subset selection in linear regression \cite{miller:90}, estimation of structures in graphical models \cite{meinshausen:annstat06}, and signal denoising \cite{donoho:siamjsc98}. In addition, solving the model-selection problem in some (but not all) cases also enables one to solve the estimation and/or the regression problem.

\subsection{Main Contributions}
\textbf{Model Selection:} One of the primary objectives of this paper is to study the problem of \emph{polynomial time, model-order agnostic model selection in a compressed setting for the general case of arbitrary (random or deterministic) design matrices and arbitrary nonzero entries of the signal}. In order to accomplish this task, we introduce in the paper two fundamental measures of coherence among the (normalized) columns $\{\mathrm{x}_i \in \C^n\}$ of the $n \times p$ design matrix $X$, namely,\footnote{Here, and throughout the rest of this paper, we assume without loss of generality that $X$ has unit $\ell_2$-norm columns. This is because deviations to this assumption can always be accounted for by appropriately scaling the entries of the data vector $\beta$ instead.}
\begin{itemize}
\item \emph{Worst-Case Coherence}: $\mu(X) \doteq \max\limits_{i,j:i \neq j} \big|\langle\mathrm{x}_i, \mathrm{x}_j\rangle\big|$, and 
%
\vspace{1em}
\item \emph{Average Coherence}: $\nu(X) \doteq \frac{1}{p-1} \max\limits_{i} \bigg|\!\sum\limits_{j:j \neq i} \langle\mathrm{x}_i, \mathrm{x}_j\rangle\bigg|.$
\end{itemize}
%
\vspace{0.5em}
Roughly speaking, worst-case coherence---which seems to have been introduced in the related literature in \cite{davis:ca97}---is a similarity measure between the columns of a design matrix: the smaller the worst-case coherence, the less similar the columns. On the other hand, average coherence---which was first introduced in a conference version of this paper \cite{bajwa:isit10}---is a measure of the spread of the columns of a design matrix within the $n$-dimensional unit ball: the smaller the average coherence, the more spread out the column vectors.

\begin{algorithm*}[t]
\caption{The One-Step Thresholding (OST) Algorithm for Model Selection}
\label{alg:OST}
\textbf{Input:} An $n \times p$ matrix $X$, a vector $y \in \C^n$, and a threshold $\lambda > 0$\\
\textbf{Output:} An estimate $\whcS \subset \{1,\dots,p\}$ of the true model $\cS$
\begin{algorithmic}
\STATE $f \leftarrow X^\tH y$ \hfill \COMMENT{Form signal proxy}
\STATE $\whcS \leftarrow \left\{i \in \{1,\dots,p\} : |f_i| > \lambda\right\}$ \hfill \COMMENT{Select model via OST}
\end{algorithmic}
\end{algorithm*}
%
Our main contribution in the area of model selection is that we make use of these two measures of coherence to propose and analyze a model-order agnostic threshold for the \emph{one-step thresholding} (OST) algorithm (see Algorithm~\ref{alg:OST}) for model selection. Specifically, we characterize in Section~\ref{sec:mod_sel} both the exact and the partial model-selection performance of OST in a \emph{non-asymptotic} setting in terms of $\mu$ and $\nu$. In particular, we establish in Section~\ref{sec:mod_sel} that if $\mu(X) \asymp n^{-1/2}$ and $\nu(X) \precsim n^{-1}$ then OST---despite being computationally primitive---can perform near-optimally for the case when either (i) the energy of any nonzero entry of $\beta$ is not too far away from the average signal energy per nonzero entry $\|\beta\|_2^2/k$ or (ii) the signal-to-noise ratio (\SNR) in the measurement system is not too high.\footnote{Recall ``\emph{Big--O}'' notation: $f(n) = O(g(n))$ (alternatively, $f(n) \precsim g(n)$) if $\exists~c_o > 0, n_o : \forall~n \geq n_o, f(n) \leq c_o g(n)$, $f(n) = \Omega(g(n))$ (alternatively, $f(n) \succsim g(n)$) if $g(n) = O(f(n))$, and $f(n) =  \Theta(g(n))$ (alternatively, $f(n) \asymp g(n)$) if $g(n) \precsim f(n) \precsim g(n)$.} Equally importantly, in contrast to some of the existing literature on model selection, this analysis in the paper holds for arbitrary values of the nonzero entries of $\beta$ and it does not require the $n \times k$ submatrices of the design matrix $X$ to have full column rank.

\textbf{Sparse-Signal Recovery:} The second main objective of this paper is to study the problem of \emph{low-complexity, model-order agnostic recovery of $k$-sparse signals with arbitrary nonzero entries in the noiseless case}. In this regard, our main contribution in the area of sparse-signal recovery is that we make use of a recent result by Tropp \cite{tropp:cras08} in Section~\ref{sec:sparse_rec} to extend our results on model selection to recovery of $k$-sparse signals using OST (see Algorithm~\ref{alg:OST_recon}). In particular, we establish in Section~\ref{sec:gabor} that Gabor frames---which are collections of time- and frequency-shifts of a nonzero seed vector (sequence) in $\C^n$---can potentially be used together with OST to exactly recover \emph{most} $k$-sparse signals with arbitrary nonzero entries as long as $k \precsim \mu^{-2}/\log{n}$ and the energy of any nonzero entry of $\beta$ is not too far away from $\|\beta\|_2^2/k$. This result then applies immediately to Gabor frames generated from the Alltop sequence \cite{alltop:tit80}. Specifically, since Gabor frames generated from the Alltop sequence have worst-case coherence $\mu = \frac{1}{\sqrt{n}}$ for any prime $n \geq 5$ \cite{strohmer:acha03}, this result implies that an Alltop Gabor frame together with OST successfully recovers most $k$-sparse signals \emph{irrespective} of the values of the nonzero entries of $\beta$ as long as $k \precsim n/\log{n}$ and and the energy of any nonzero entry of $\beta$ is not too far away from $\|\beta\|_2^2/k$.
%
\begin{algorithm*}[t]
\caption{The One-Step Thresholding (OST) Algorithm for Sparse-Signal Recovery}
\label{alg:OST_recon}
\textbf{Input:} An $n \times p$ matrix $X$, a vector $y \in \C^n$, and a threshold $\lambda > 0$\\
\textbf{Output:} An estimate $\widehat{\beta} \in \C^p$ of the true sparse signal $\beta$
\begin{algorithmic}
\STATE $\widehat{\beta} \leftarrow \bzero$ \hfill \COMMENT{Initialize}
\STATE $f \leftarrow X^\tH y$ \hfill \COMMENT{Form signal proxy}
\STATE $\cI \leftarrow \left\{i \in \{1,\dots,p\} : |f_i| > \lambda\right\}$ \hfill \COMMENT{Select indices via OST}
\STATE $\widehat{\beta}_{\cI} \leftarrow (X_\cI)^\dagger y$ \hfill \COMMENT{Recover signal via least-squares}
\end{algorithmic}
\end{algorithm*}

\subsection{Relationship to Previous Work}
The problems of model selection and sparse-signal recovery in general and the use of OST (also known as \emph{simple thresholding} \cite{donoho:cpam06} and \emph{marginal regression} \cite{genovese:sub09}) to solve these problems in particular have a rich history in the literature. In the context of model selection in the compressed setting, Mallow's $C_p$ selection procedure \cite{mallows:techno73} and the \emph{Akaike information criterion} (AIC) \cite{akaike:tac74}---both of which essentially attempt to solve a complexity-regularized version of the least-squares criterion---are considered to be seminal works, and are known to perform well empirically as well as theoretically; see, e.g., \cite{massart:ecm05} and the references therein. These two procedures have been modified by numerous researchers over the years in order to improve their performance---the most notable variants being the \emph{Bayesian information criterion} (BIC) \cite{schwarz:annstat78} and the \emph{risk inflation criterion} (RIC) \cite{foster:annstat94}. Solving model-selection procedures such as $C_p$, AIC, BIC, and RIC, however, is known to be an NP-hard problem \cite{natarajan:siamjc95} even if the true model order $k$ is made available to these procedures.

In order to overcome the computational intractability of these model-selection procedures, several methods based on convex optimization have been proposed by various researchers in recent years. Among these proposed methods, the lasso \cite{tibshirani:jrss96} has arguably become the standard tool for model selection, which can be partly attributed to the theoretical guarantees provided for the lasso in \cite{meinshausen:annstat06,zhao:jmlr06,wainwright:tit09,candes:annstat09}. In particular, the results reported in \cite{meinshausen:annstat06,zhao:jmlr06} establish that the lasso asymptotically identifies the correct model under certain conditions on the design matrix $X$ and the sparse vector $\beta$. Later, Wainwright in \cite{wainwright:tit09} strengthens the results of \cite{meinshausen:annstat06,zhao:jmlr06} and makes explicit the dependence of exact model selection using the lasso on the smallest (in magnitude) nonzero entry of $\beta$. However, apart from the fact that the results reported in \cite{meinshausen:annstat06,zhao:jmlr06,wainwright:tit09} are for exact model selection and are only asymptotic in nature, the main limitation of these works is that explicit verification of the conditions (such as the \emph{irrepresentable condition} of \cite{zhao:jmlr06} and the \emph{incoherence condition} of \cite{wainwright:tit09}) that a generic design matrix $X$ needs to satisfy is computationally intractable for $k \succsim \mu^{-1}$. The most general (and non-asymptotic) model-selection results using the lasso for arbitrary design matrices have been reported in \cite{candes:annstat09}. Specifically, Cand\`{e}s and Plan have established in \cite{candes:annstat09} that the lasso correctly identifies most models with probability $1 - O(p^{-1})$ under certain conditions on the smallest nonzero entry of $\beta$ provided: (i) the spectral norm (the largest singular value) and the worst-case coherence of $X$ are not too large, and (ii) the values of the nonzero entries of $\beta$ are independent and statistically symmetric around zero. Despite these recent theoretical triumphs of the lasso, it is still desirable to study alternative solutions to the problem of polynomial time, model-order agnostic model selection in a compressed setting. This is because:\footnote{During the course of revising this paper we also became aware of \cite{saligrama:10sub}, which proposes a thresholded variant of \emph{basis pursuit} \cite{donoho:siamjsc98} for sparsity pattern recovery using Gaussian design matrices. However, the results reported in \cite{saligrama:10sub} are limited because of similar issues and because of the requirement that the magnitude of the smallest nonzero entry of $\beta$ be known to the algorithm.}
\begin{enumerate}
\item Lasso requires the minimum singular values of the submatrices of $X$ corresponding to the true models to be bounded away from zero \cite{meinshausen:annstat06,zhao:jmlr06,wainwright:tit09,candes:annstat09}. While this is a plausible condition for the case when one is interested in estimating $\beta$, it is arguable whether this condition is necessary for the case of model selection.
\item The current literature on model selection using the lasso lacks guarantees beyond $k \succsim \mu^{-1}$ for the case of generic design matrices and arbitrary nonzero entries. In particular, given an arbitrary design matrix $X$, \cite{meinshausen:annstat06,zhao:jmlr06,wainwright:tit09,candes:annstat09} do not provide any guarantees beyond $k \succsim \sqrt{n}$ for even the simple case of $\beta \in \R^p_+$.
\item The computational complexity of the lasso for generic design matrices tends to be $O(p^3 + np^2)$ \cite{genovese:sub09}. This makes the lasso computationally demanding for large-scale model-selection problems.
\end{enumerate}

Recently, a few researchers have raised somewhat similar concerns about the lasso and revisited the much older (and oft-forgotten) method of thresholding for model selection \cite{schnass:spl07,fletcher:tit09,reeves:asilomar09,genovese:sub09}, which has computational complexity of $O(np)$ only and which is known to be nearly optimal for $p \times p$ orthonormal design matrices \cite{donoho:biomet94}. Algorithmically, this makes our approach to model selection similar to that of \cite{schnass:spl07,fletcher:tit09,reeves:asilomar09,genovese:sub09}. Nevertheless, the OST algorithm presented in this paper differs from \cite{schnass:spl07,fletcher:tit09,reeves:asilomar09,genovese:sub09} in five key aspects:
\begin{enumerate}
\item \emph{Model-Order Agnostic Model Selection:} Unlike \cite{schnass:spl07,fletcher:tit09,reeves:asilomar09,genovese:sub09}, the OST algorithm presented in this paper is completely agnostic to both the true model order $k$ and any estimate of $k$.
\item \emph{Generic Design Matrices and Arbitrary Nonzero Entries:} The results reported in this paper hold for arbitrary (random or deterministic) design matrices and do not assume any statistical prior on the values of the nonzero entries of $\beta$ even when $k$ scales linearly with $n$. In contrast, \cite{fletcher:tit09} only studies the problem of Gaussian design matrices whereas the most influential results reported in \cite{schnass:spl07,reeves:asilomar09,genovese:sub09} assume that the values of the nonzero entries of $\beta$ are independent and statistically symmetric around zero.
\item \emph{Verifiable Sufficient Conditions:} In contrast to \cite{schnass:spl07, fletcher:tit09, reeves:asilomar09, genovese:sub09}, we relate the model-selection performance of OST to two global parameters of $X$, namely, $\mu$ and $\nu$, which are trivially computable in polynomial time: $\mu(X) = \|X^\tH X - I\|_{\max}$ and $\nu(X) = \frac{1}{p-1}\|(X^\tH X - I)\bone\|_\infty$.
\item \emph{Non-Asymptotic Theory:} Similar to \cite{fletcher:tit09,reeves:asilomar09,genovese:sub09}, the analysis in this paper can be used to establish that OST achieves (asymptotically) consistent model selection under certain conditions. However, the results reported in this paper are completely non-asymptotic in nature (with explicit constants) and thereby shed light on the rate at which OST achieves consistent model selection.
\item \emph{Partial Model Selection:} In addition to the exact model-selection performance of OST, we also characterize in the paper its partial model-selection performance. In this regard, we establish that the \emph{universal threshold} proposed in Section~\ref{sec:mod_sel} for OST guarantees $\whcS \subset \cS$ with high probability and we quantify the cardinality of the estimate $\whcS$. On the other hand, both \cite{schnass:spl07} and \cite{fletcher:tit09} study only exact model selection, whereas \cite{genovese:sub09,reeves:asilomar09} study approximate (though not partial) model selection only for Gaussian design matrices \cite{genovese:sub09} and assuming Gaussian (resp. statistical) priors on the nonzero entries of $\beta$ \cite{reeves:asilomar09} (resp. \cite{genovese:sub09}).
\end{enumerate}

We conclude this discussion of model selection by making three important remarks. First, to the best of our knowledge, Donoho in \cite[Theorem~7.2]{donoho:cpam06} reported some of the earliest known results for thresholding in the compressed setting. Nevertheless, the conclusion drawn in \cite{donoho:cpam06} was that thresholding is feasible for model selection as long as $k \precsim \mu^{-1}$, the so-called ``square root bottleneck.'' Second, the structure of OST and the model-order agnostic threshold of this paper enable us to carry out \emph{localized model selection}. Specifically, if one is provided at the time of recovery with a set $\cT$ such that $\cT \supset \cS$ then the threshold proposed in this paper enables one to carry out model selection using the submatrix $X_\cT$ instead of $X$, thereby reducing the complexity of OST from $O(np)$ to $O(n|\cT|)$. Third, the results reported in this paper hold for any $n \leq p$ and, in particular, the universal threshold proposed here for model selection reduces to the universal threshold proposed by Donoho and Johnstone \cite{donoho:biomet94} for $p \times p$ orthonormal design matrices. In this sense, some of the results reported in \cite{donoho:biomet94} can also be thought of as special instances of the results reported in this paper.

Finally, in the context of sparse-signal recovery in the compressed setting, there exists now a large body of literature that studies this problem under the rubric of \emph{compressed sensing} \cite{spm:cs08}. However, convex optimization procedures such as basis pursuit (BP) \cite{donoho:siamjsc98}, Dantzig selector \cite{candes:annstat07}, and lasso---although known for their ability to recover sparse signals under a variety of conditions---are ill-suited for large-scale problems because of their computational complexity. On the other hand, low-complexity iterative algorithms such as matching pursuit \cite{mallat:tsp93}, subspace pursuit \cite{dai:tit09}, CoSaMP \cite{needell:acha09}, and iterative hard thresholding \cite{blumensath:acha09}, and combinatorial algorithms based on group testing such as HHS pursuit \cite{gilbert:stoc07} and Fourier samplers \cite{gilbert:stoc02,iwen:soda08} have been shown to perform well either only for some special classes of design matrices \cite{gilbert:stoc02,gilbert:stoc07,iwen:soda08} or for design matrices that satisfy the \emph{restricted isometry property} (RIP) \cite{candes:cras08}. Nevertheless, explicitly verifying that $X$ satisfies the RIP of order $k \succsim \mu^{-1}$ is computationally intractable; in particular, since we have from the Welch bound \cite{welch:tit74} that $\mu^{-1} \precsim \sqrt{n}$ for $p \gg 1$, the guarantees provided in \cite{dai:tit09, needell:acha09, blumensath:acha09} for the case of generic design matrices at best hold only for $k$-sparse signals with $k \precsim \sqrt{n}$.

In contrast, and motivated by the need to have verifiable sufficient conditions for low-complexity algorithms and arbitrary values of the nonzero entries of $\beta$ even when $k \succsim \sqrt{n}$, we extend in Section~\ref{sec:sparse_rec} our results on model selection using OST and characterize the performance of Algorithm~\ref{alg:OST_recon} in terms of three global parameters of the design matrix $X$: $\mu(X)$, $\nu(X)$, and $\|X\|_2$. In particular, a highlight of this part of the paper is that we partially strengthen the results of Pfander et al. \cite{pfander:tsp08} and Herman and Strohmer \cite{herman:tsp09} by establishing that Gabor frames generated from the Alltop sequence can be used along with OST to recover most $k$-sparse signals belonging to certain classes even when $k \succsim \sqrt{n}$. It is worth pointing out here that both \cite{pfander:tsp08,herman:tsp09} also establish that Alltop Gabor frames can recover most $k$-sparse signals---albeit using BP---even when $k \succsim \sqrt{n}$. Nevertheless, the basic difference between \cite{pfander:tsp08,herman:tsp09} and the work presented here is that \cite{pfander:tsp08,herman:tsp09} require the phases of the nonzero entries of $\beta$ to be statistically independent and uniformly distributed on the unit torus whereas we do not assume any statistical prior on the values of the nonzero entries of $\beta$. Note in particular that, just like the lasso result in \cite{candes:annstat09}, the results reported in \cite{pfander:tsp08,herman:tsp09} for Alltop Gabor frames consequently do not provide any guarantees beyond $k \succsim \sqrt{n}$ for even the simple case of $\beta \in \R^p_+$. This difference between the BP-based recovery guarantees presented in \cite{pfander:tsp08,herman:tsp09} (which are essentially based on \cite{tropp:acha08}) and the OST-based recovery guarantees provided in this paper is also illustrated using a Venn diagram in Fig.~\ref{fig:OST_vs_BP} for
\emph{unimodal signals} (defined as: $|\beta_i| \approx c$ for some arbitrary $c > 0$ and for all $i \in \cS$).
\begin{figure*}[t]
\centering%
\includegraphics[scale=0.4]{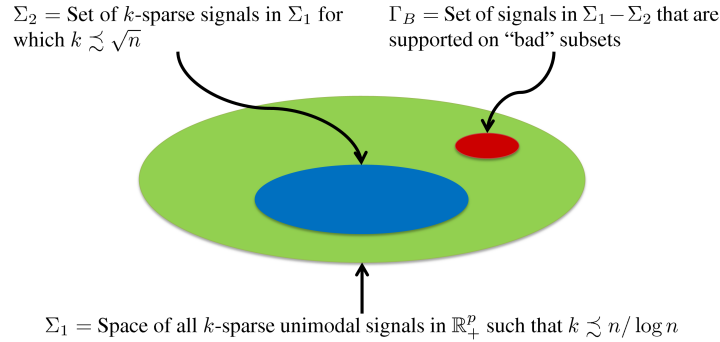}
\caption{A Venn digram used to illustrate the major difference between the BP-based recovery guarantees and the OST-based recovery guarantees for $k$-sparse unimodal signals in $\R^p_+$ measured using Alltop Gabor frames. The OST algorithm is guaranteed to recover $\beta \in \Sigma_1 - \Gamma_{B}$. But BP, unlike OST, is only guaranteed to recover $\beta \in \Sigma_2$ in this case.}%
\label{fig:OST_vs_BP}%
\end{figure*}

\subsection{Notation}
The following notation is used throughout the rest of this paper. We use lowercase letters to denote scalars and vectors, while we use uppercase letters to denote matrices. We also use $\bzero$, $\bone$, and $I$ to denote the all-zeros vector, the all-ones vector, and the identity matrix, respectively. In addition, we use $\|v\|_p$ to denote the usual $\ell_p$-norm of a vector $v$, while we use $A^\dag$, $\|A\|_2$, and $\|A\|_{\max}$ to denote the Moore--Penrose pseudoinverse, the spectral norm, and the maximum magnitude of any entry of a matrix $A$, respectively. Further, we use $(\cdot)^\tT$ and $(\cdot)^\tH$ to denote the operations of transposition and conjugate transposition, respectively, while we use $\langle\cdot,\cdot\rangle$ to denote inner product that is conjugate linear in the first argument. Finally, given a set $\cI$, we use $v_{\cI}$ to denote the part of a vector $v$ corresponding to the indices in $\cI$ and $A_\cI$ to denote the submatrix obtained by collecting the $|\cI|$ columns of a matrix $A$ corresponding to the indices in $\cI$.

\subsection{Organization}
The rest of this paper is organized as follows. In Section~\ref{sec:mod_sel}, we propose a model-order agnostic threshold for the OST algorithm and characterize both the exact and the partial model-selection performance of OST. In Section~\ref{sec:sparse_rec}, we extend our results on model selection and characterize the sparse-signal recovery performance of OST. In Section~\ref{sec:gabor}, we specialize the model-selection and the sparse-signal recovery results of the previous sections to Gabor frames. Finally, we provide proofs of the main results of this paper in Section~\ref{sec:proofs} and conclude with a discussion of the limitations and extensions of our results in Section~\ref{sec:disc}.

\section{Model Selection Using One-Step Thresholding}\label{sec:mod_sel}
\subsection{Assumptions}
Before proceeding with presenting our results on model selection using OST, we need to be mathematically precise about our problem formulation. To this end, we begin by reconsidering the measurement model $y = X \beta + \eta$ and assume that $X$ is an $n \times p$ real- or complex-valued design matrix having unit $\ell_2$-norm columns, $\beta \in \C^p$ is a $k$-sparse signal $(\|\beta\|_0 \leq k)$, and $k < n \leq p$. Here, we allow $X$ to be either a random or a deterministic design matrix, while we take $\eta$ to be a complex additive white Gaussian noise vector. It is worth mentioning here though that Gaussianity of $\eta$ is just a simplified assumption for the sake of this exposition; in particular, the results presented in this section are readily generalizable to other noise distributions as well as perturbations having bounded $\ell_2$-norms. Finally, the main assumption that we make here is that the true model $\cS \doteq \{i \in \{1,\dots,p\}:|\beta_i| > 0\}$ is a uniformly random $k$-subset of $\{1,\dots,p\}$. In other words, we have a uniform prior on the \emph{support} of the data vector $\beta$.

\subsection{Main Results}\label{ssec:ms_results}
Intuitively speaking, successful model selection requires the columns of the design matrix to be \emph{incoherent}. In the case of the lasso, this notion of incoherence has been quantified in \cite{zhao:jmlr06} and \cite{wainwright:tit09} in terms of the ``irrepresentable condition'' and the ``incoherence condition,'' respectively (see also \cite{candes:annstat09}). In contrast to earlier work on model selection, however, we formulate this idea of incoherence in terms of the \emph{coherence property}.
\begin{definition}[The Coherence Property]
An $n \times p$ design matrix $X$ having unit $\ell_2$-norm columns is said to obey the coherence property if the following two conditions hold:
\begin{align}
\label{eqn:cp1}
\tag{CP-1} \mu(X) &\leq \frac{0.1}{\sqrt{2 \log{p}}}, \quad \text{and}\\
\label{eqn:cp2}
\tag{CP-2} \nu(X) &\leq \frac{\mu}{\sqrt{n}}\,.
\end{align}
\end{definition}
\noindent In words, \eqref{eqn:cp1} roughly states that the columns of $X$ are not too similar, while \eqref{eqn:cp2} roughly states that the columns of $X$ are somewhat distributed within the $n$-dimensional unit ball. Note that the coherence property is superior to other measures of incoherence such as the irrepresentable condition in two key aspects. First, it does not require the singular values of the submatrices of $X$ to be bounded away from zero. Second, it can be easily verified in polynomial time since it simply requires checking that $\|X^\tH X - I\|_{\max} \leq (200 \log{p})^{-1/2}$  and $\|(X^\tH X - I)\bone\|_\infty \leq (p-1)n^{-1/2}\|X^\tH X - I\|_{\max}$.

Below, we describe the implications of the coherence property for both the exact and the partial model-selection performance of OST. Before proceeding further, however, it is instructive to first define some fundamental quantities pertaining to the problem of model selection as follows:
\begin{align*}
 \beta_{\min} &\doteq \min_{i\in\cS}|\beta_i|\,, \qquad & \MAR &\doteq \frac{\beta_{\min}^2}{\|\beta\|_2^2/k}\,,\\
 \SNR_{\min} &\doteq \frac{\beta_{\min}^2}{\E[\|\eta\|_2^2]/k}\,, \qquad & \SNR &\doteq \frac{\|\beta\|_2^2}{\E[\|\eta\|_2^2]}\,.
\end{align*}
In words, $\beta_{\min}$ is the magnitude of the smallest nonzero entry of $\beta$, while $\MAR$---which is termed as \emph{minimum-to-average ratio} \cite{fletcher:tit09}---is the ratio of the \emph{energy in the smallest nonzero entry} of $\beta$ and the \emph{average signal energy per nonzero entry} of $\beta$. Likewise, $\SNR_{\min}$ is the ratio of the energy in the smallest nonzero entry of $\beta$  and the average \emph{noise} energy per nonzero entry, while $\SNR$ simply denotes the usual signal-to-noise ratio in the system. It is also worth pointing out here the relationship between $\SNR_{\min}$ and $\SNR$ and $\MAR$; specifically, it is easy to see that $\SNR_{\min} = \SNR\cdot\MAR$. We are now ready to state the first main result of this paper that concerns the performance of OST in terms of exact model selection.
\begin{theorem}[Exact Model Selection Using OST]\label{thm:ems_cp}
Suppose that the design matrix $X$ satisfies the coherence property and let $\eta$ be distributed as $\cCN(\bzero,\sigma^2 I)$. Next, choose the threshold $\lambda = \max\Big\{\frac{1}{t}10 \mu \sqrt{n \cdot \SNR}, \frac{1}{1-t} \sqrt{2}\Big\}\sqrt{2\sigma^2\log{p}}$ for any $t \in (0,1)$. Then, if we write $\mu(X)$ as $\mu = c_1 n^{-1/\gamma}$ for some $c_1 > 0$ (which may depend on $p$) and $\gamma \in \{0\}\cup [2,\infty)$, the OST algorithm (Algorithm~\ref{alg:OST}) satisfies $\Pr(\whcS \not= \cS) \leq 6p^{-1}$ provided $p \geq 128$ and the number of measurements satisfies
\begin{align}
\nonumber
 n &> \max \Bigg\{2k\log{p},\frac{8(1-t)^{-2}}{\SNR_{\min}}2k\log{p},\bigg(\frac{c_2t^{-2}}{\MAR}2k\log{p}\bigg)^{\gamma/2}\Bigg\}\\
\label{eqnthm:ems_1}
 &\equiv \max \Bigg\{2k\log{p},\frac{8(1-t)^{-2}}{\SNR\cdot\MAR}2k\log{p},\bigg(\frac{c_2t^{-2}}{\MAR}2k\log{p}\bigg)^{\gamma/2}\Bigg\}.
\end{align}
Here, the quantity $c_2 > 0$ is defined as $c_2 \doteq (20\,c_1)^2$, while the probability of failure is with respect to the true model $\cS$ and the complex Gaussian noise vector $\eta$.
\end{theorem}
The proof of this theorem is provided in Section~\ref{sec:proofs}. Note that the parameter `$t$' in Theorem~\ref{thm:ems_cp} can always be fixed a priori (say $t = 1/2$) without affecting the scaling relation in \eqref{eqnthm:ems_1}. In practice, however, $t$ should be chosen so as to reduce the total number of measurements needed to ensure successful model selection; the optimal choice of $t$ in this regard is $t_{opt} = \argmin\limits_{t} \left(\max \bigg\{ \frac{8(1-t)^{-2}}{\SNR\cdot\MAR}2k\log{p}, \Big(\frac{c_2t^{-2}}{\MAR}2k\log{p}\Big)^{\gamma/2}\bigg\}\right)$. Notice also that Theorem~\ref{thm:ems_cp} is best suited for applications where one is interested in quantifying the minimum number of measurements needed to guarantee exact model selection for a given class of signals. Alternatively, it might be the case in some other applications that the problem dimensions are fixed and one is instead interested in specifying the class of signals that leads to successful model selection. The following variant of Theorem~\ref{thm:ems_cp} is best suited in such situations.
\begin{theorem}\label{thm:ems_cp2}
Suppose that the design matrix $X$ satisfies the coherence property and let the noise vector $\eta$ be distributed as $\cCN(\bzero,\sigma^2 I)$. Next, let $p \geq 128$ and choose the threshold $\lambda = \max\Big\{\frac{1}{t}10 \mu \sqrt{n \cdot \SNR}, \frac{1}{1-t} \sqrt{2}\Big\}\sqrt{2\sigma^2\log{p}}$ for any $t \in (0,1)$. Then the OST algorithm (Algorithm~\ref{alg:OST}) satisfies $\Pr(\whcS \not= \cS) \leq 6p^{-1}$ as long as we have that $k \leq n/(2\log{p})$ and
\begin{align}
\label{eqnthm:ems_2}
\MAR > \max \Bigg\{8(1-t)^{-2} \left(\frac{2k\log{p}}{n \cdot \SNR}\right), 400t^{-2} \left(\frac{2k\log{p}}{\mu^{-2}}\right)\Bigg\}.
\end{align}
Here, the probability of failure is with respect to the true model $\cS$ and the complex Gaussian noise vector $\eta$.
\end{theorem}

\begin{algorithm*}[t]
\caption{The Sorted One-Step Thresholding (SOST) Algorithm for Model Selection}
\label{alg:SOST}
\textbf{Input:} An $n \times p$ matrix $X$, a vector $y \in \C^n$, and model order $k$\\
\textbf{Output:} An estimate $\whcS \subset \{1,\dots,p\}$ of the true model $\cS$
\begin{algorithmic}
\STATE $f \leftarrow X^\tH y$ \hfill \COMMENT{Form signal proxy}
\STATE $(\cI, f_s) \leftarrow \text{SORT}\Big(\big(\{1,\dots,p\},f\big)\Big)$ \hfill \COMMENT{Sort the signal proxy} 
\STATE $\whcS \leftarrow \cI[1:k]$ \hfill \COMMENT{Select model via OST}
\end{algorithmic}
\end{algorithm*}
%
Note that the proof of Theorem~\ref{thm:ems_cp2} follows directly from the proof of Theorem~\ref{thm:ems_cp}. There are a few important remarks that need to be made at this point concerning the threshold proposed in Theorem~\ref{thm:ems_cp} and Theorem~\ref{thm:ems_cp2} for the OST algorithm. First, it is easy to see that the proposed threshold is completely agnostic to the model order $k$ and only requires knowledge of the $\SNR$ and the noise variance. Second, some of the bounds in the proof of Theorem~\ref{thm:ems_cp} and extensive simulations suggest that the absolute constant $10$ in the proposed threshold is somewhat conservative and can be reduced through the use of more sophisticated analytical tools (also see Section~\ref{sec:disc}). Finally, while estimating the true model order $k$ tends to be harder than estimating the $\SNR$ and the noise variance $\sigma^2$ in majority of the situations, it might be the case that estimating $k$ is easier in some applications. It is better in such situations to work with a slight variant of the OST algorithm (see Algorithm~\ref{alg:SOST}) that relies on knowledge of the model order $k$ instead and returns an estimate $\whcS$ corresponding to the $k$ largest (in magnitude) entries of $X^H y$. We characterize the performance of this algorithm---which we term as \emph{sorted one-step thresholding} (SOST) algorithm---in terms of the following theorem.
\begin{theorem}[Exact Model Selection Using SOST]\label{thm:sost_ems_cp}
Suppose that the design matrix $X$ satisfies the coherence property and let $\eta$ be distributed as $\cCN(\bzero,\sigma^2 I)$. Next, write $\mu(X)$ as $\mu = c_1 n^{-1/\gamma}$ for some $c_1 > 0$ (which may depend on $p$) and $\gamma \in \{0\}\cup [2,\infty)$. Then the SOST algorithm (Algorithm~\ref{alg:SOST}) satisfies $\Pr(\whcS \not= \cS) \leq 6p^{-1}$ as long as $p \geq 128$ and the number of measurements satisfies
\begin{align}
\nonumber
 n &> \min_{t \in (0,1)} \max \Bigg\{2k\log{p},\frac{8(1-t)^{-2}}{\SNR_{\min}}2k\log{p},\bigg(\frac{c_2t^{-2}}{\MAR}2k\log{p}\bigg)^{\gamma/2}\Bigg\}\\
\label{eqnthm:sost_ems_1}
 &\equiv \min_{t \in (0,1)} \max \Bigg\{2k\log{p},\frac{8(1-t)^{-2}}{\SNR\cdot\MAR}2k\log{p},\bigg(\frac{c_2t^{-2}}{\MAR}2k\log{p}\bigg)^{\gamma/2}\Bigg\}.
\end{align}
Here, the quantity $c_2 > 0$ is as defined in Theorem~\ref{thm:ems_cp}, while the probability of failure is with respect to the true model $\cS$ and the complex Gaussian noise vector $\eta$.
\end{theorem}

The proof of this theorem is just a slight variant of the proof of Theorem~\ref{thm:ems_cp} and is therefore omitted here. A few remarks are in order now concerning OST and SOST. First, the computational complexity of SOST is comparable with that of OST since efficient sorting algorithms (such as heap sort) tend to have computational complexity of $O(p\log{p})$ only. Second, \eqref{eqnthm:ems_1} and \eqref{eqnthm:sost_ems_1} suggest that knowledge of the true model order $k$ allows SOST to perform better than OST in situations where the threshold parameter $t$ is fixed a priori (cf.~Theorem~\ref{thm:ems_cp}). In this sense, SOST should be preferred over OST for exact model selection \emph{provided} one has accurate knowledge of the true model order $k$. On the other hand, OST should be the algorithm of choice for model-selection problems where it is difficult to obtain a reliable estimate of the true model order. We conclude this discussion by rephrasing Theorem~\ref{thm:sost_ems_cp} for SOST along the lines of Theorem~\ref{thm:ems_cp2} for OST.
\begin{theorem}\label{thm:sost_ems_cp2}
Suppose that the design matrix $X$ satisfies the coherence property. Next, let $p \geq 128$ and let the noise vector $\eta$ be distributed as $\cCN(\bzero,\sigma^2 I)$. Then the SOST algorithm (Algorithm~\ref{alg:SOST}) satisfies $\Pr(\whcS \not= \cS) \leq 6p^{-1}$ provided $k \leq n/(2\log{p})$ and
\begin{align}
\label{eqnthm:sost_ems_2}
\MAR > \min_{t \in (0,1)} \max \Bigg\{8(1-t)^{-2} \left(\frac{2k\log{p}}{n \cdot \SNR}\right), 400t^{-2} \left(\frac{2k\log{p}}{\mu^{-2}}\right)\Bigg\}.
\end{align}
Here, the probability of failure is with respect to the true model $\cS$ and the complex Gaussian noise vector $\eta$.
\end{theorem}

The final result that we present in this section concerns the partial model-selection performance of OST. Specifically, note that our focus in this section has so far been on specifying conditions for either the number of measurements or the $\MAR$ of the signal that ensure exact model selection. In many real-world applications, however, the parameters of the problem are fixed and it is not always possible to ensure that either the number of measurements or the $\MAR$ of the signal satisfy the aforementioned conditions. A natural question to ask then is whether the OST algorithm completely fails in such circumstances or whether any guarantees can still be provided for its performance. We address this aspect of the OST algorithm in the following and show that, even if the $\MAR$ of $\beta$ is very small, OST has the ability to identify the locations of the nonzero entries of $\beta$ whose energies are greater than both the noise power and the average signal energy per nonzero entry. In order to make this notion mathematically precise, we first define the $m$\emph{-th largest-to-average ratio} ($\LAR_m$) of $\beta$ as the ratio of the \emph{energy in the $m$-th largest (in magnitude) nonzero entry of $\beta$} and the average signal energy per nonzero entry of $\beta$; that is, $$\LAR_m \doteq \frac{|\beta_{(m)}|^2}{\|\beta\|_2^2/k}$$ where $\beta_{(m)}$ denotes the $m$-th largest nonzero entry of $\beta$ (note that $\MAR \equiv \LAR_k$). We are now ready to specify the partial model-selection performance of the OST algorithm.
\begin{theorem}[Partial Model Selection Using OST]\label{thm:pms_cp2}
Suppose that the design matrix $X$ satisfies the coherence property. Next, let $p \geq 128$ and $\eta$ be distributed as $\cCN(\bzero,\sigma^2 I)$. Finally, fix a parameter $t \in (0,1)$ and choose the threshold $\lambda = \max\Big\{\frac{1}{t}10 \mu \sqrt{n \cdot \SNR}, \frac{1}{1-t} \sqrt{2}\Big\}\sqrt{2\sigma^2\log{p}}$. Then, under the assumption that $k \leq n/(2\log{p})$, the OST algorithm (Algorithm~\ref{alg:OST}) guarantees with probability exceeding $1 - 6p^{-1}$ that $\whcS \subset \cS$ and $\big|\cS - \whcS\big| \leq (k - M)$, where $M$ is the largest integer for which the following inequality holds:
\begin{align}
\label{eqnthm:pms_2}
\LAR_{M} > \max \Bigg\{8(1-t)^{-2} \left(\frac{2k\log{p}}{n \cdot \SNR}\right), 400t^{-2} \left(\frac{2k\log{p}}{\mu^{-2}}\right)\Bigg\}.
\end{align}
Here, the probability of failure is with respect to the true model $\cS$ and the complex Gaussian noise vector $\eta$.
\end{theorem}
The proof of this theorem, which relies to a great extent on the proof of Theorem~\ref{thm:ems_cp}, is provided in Section~\ref{sec:proofs}. We conclude this section by pointing out that no counterpart of Theorem~\ref{thm:pms_cp2} exists for the SOST algorithm since we can never have $\whcS \subset \cS$ in that case because of the nature of the algorithm.

\subsection{Discussion}
The results reported in this section can be best put into perspective by considering some specific model-selection problems that are commonly studied in the literature and juxtaposing our results with the ones reported in previous works. The rest of this section is devoted to such comparison purposes.

\subsubsection{Gaussian Design Matrices}
Matrices with independent and identically distributed (i.i.d.) $\cN(0,1/n)$ entries (i.e., Gaussian matrices) are perhaps the most widely assumed design matrices in the model-selection literature. In order to specialize our results to Gaussian design matrices, we first need to specify the worst-case coherence $\mu$ and the average coherence $\nu$ of i.i.d. Gaussian matrices. The first lemma that we have in this regard follows immediately from Proposition~\ref{prop:haupt_ip} in Appendix~\ref{app:conc_ineq} through a simple union bound argument.
\begin{lemma}[Worst-Case Coherence of Gaussian Matrices]\label{lem:mu_gaussian}
Let $X$ be an $n \times p$ design matrix with i.i.d. $\cN(0,1/n)$ entries. Then, as long as $n \geq 60\log{p}$, we have that $\mu(X) \leq \sqrt{\frac{15\log{p}}{n}}$ with probability exceeding $1 - 2p^{-1}$.
\end{lemma}
\begin{remark}
A cautious reader might argue here that Lemma~\ref{lem:mu_gaussian} only provides an upperbound on the worst-case coherence of Gaussian design matrices. Nevertheless, the results (and the definition of the coherence property) presented earlier in this section remain valid if one replaces $\mu(X)$ with an upperbound $\bar{\mu}(X)$ on $\mu(X)$.
\end{remark}
\begin{lemma}[Average Coherence of Gaussian Matrices]\label{lem:nu_gaussian}
Let $X$ be an $n \times p$ design matrix with i.i.d. $\cN(0,1/n)$ entries. Then, as long as $p > n \geq 60\log{p}$, we have that $\nu(X) \leq \frac{\sqrt{15\log{p}}}{n}$ with probability exceeding $1 - 2p^{-2}$.
\end{lemma}
\begin{proof}
The proof of this lemma is also a direct consequence of Proposition~\ref{prop:haupt_ip} in Appendix~\ref{app:conc_ineq}. Specifically, fix an index $i \in \{1,\dots,p\}$ and define $\tilde{\mathrm{x}}_i \doteq \frac{1}{\sqrt{p-1}} \sum_{j \not= i} \mathrm{x}_j$. Then it is easy to see that $\tilde{\mathrm{x}}_i$ is distributed as $\cN(\bzero, I/n)$ and it is independent of $\mathrm{x}_i$. Therefore Proposition~\ref{prop:haupt_ip} in Appendix~\ref{app:conc_ineq} implies through a simple union bound argument that $\max_i |\langle\mathrm{x}_i, \tilde{\mathrm{x}}_i\rangle| \leq \sqrt{\frac{15\log{p}}{n}}$ with probability exceeding $1 - 2p^{-2}$ as long as $n \geq 60\log{p}$. The proof of the lemma now follows from the fact that $p > n$ and $\nu(X) = \frac{1}{\sqrt{p-1}} \max_i |\langle\mathrm{x}_i, \tilde{\mathrm{x}}_i\rangle|$.
\end{proof}
Lemma~\ref{lem:mu_gaussian} and Lemma~\ref{lem:nu_gaussian} establish that Gaussian design matrices satisfy the coherence property with high probability as long as $n \succsim (\log{p})^2$. Theorem~\ref{thm:ems_cp} (resp. Theorem~\ref{thm:sost_ems_cp}) therefore implies that OST (resp. SOST) correctly identifies the exact model with probability exceeding $1 - O(p^{-1})$ as long as $n \succsim \max\Big\{1, \frac{1}{\SNR\cdot\MAR}, \frac{\log{p}}{\MAR}\Big\} k\log{p}$. In particular, this suggests that if either $\MAR(\beta) = \Theta(1)$ or $\SNR = O(1)$ then OST leads to successful model selection with high probability provided $n \succapprox \max\Big\{1, \frac{1}{\SNR\cdot\MAR}\Big\} k\log{p}$.\footnote{Here, and throughout the rest of this paper, we use the shorthand notation $f(n) \succapprox g(n)$ (resp. $f(n) \precapprox g(n)$) to indicate that $f(n) \succsim g(n)$ (resp. $f(n) \precsim g(n)$) modulo a logarithmic factor.} On the other hand, one of the best known results for model selection using the maximum likelihood algorithm requires that $n \succsim \max\Big\{\frac{k\log{(p-k)}}{\SNR\cdot\MAR}, k\log{(p/k)}\Big\}$ \cite{wainwright:tit09b} (also see \cite{fletcher:tit09,akcakaya:tit10}). This establishes that OST (and its variants) performs near-optimally for Gaussian design matrices provided (i) the $\SNR$ in the measurement system is not too high or (ii) the energy of any nonzero entry of $\beta$ is not too far away from the average energy $\|\beta\|_2^2/k$ and $k$ scales sublinearly with $p$.
\begin{remark}
It is worth pointing out here that somewhat similar results can also be obtained for sub-Gaussian design matrices (i.e., matrices with entries given by i.i.d. bounded random variables, etc.) using standard concentration inequalities. Note also that the preceding discussion regarding Gaussian design matrices strengthens the results of Fletcher et al. \cite{fletcher:tit09} concerning asymptotic (Gaussian) model selection using thresholding (cf.~\cite[Theorem~2]{fletcher:tit09}).
\end{remark}

\subsubsection{Lasso versus OST}
Historically, OST (and its variants) is preferred over the lasso because of its low computational complexity. The results reported in this paper, however, bring forth another important aspect of OST (also see \cite{genovese:sub09}): \emph{OST can lead to successful model selection even when the lasso fails}. Specifically, note that the lasso solution is not even guaranteed to be unique if the minimum singular value of the submatrix of $X$ corresponding to the true model is not bounded away from zero (see, e.g., \cite{zhao:jmlr06, wainwright:tit09}). On the other hand, OST does not require the aforementioned condition for model selection. Note that this is in part due to the fact that model selection using the lasso is in fact a byproduct of signal reconstruction, whereas the aforementioned OST results do not guarantee signal reconstruction without imposing additional constraints on $X$. In other words, we have established in the paper that \emph{model selection is inherently an easier problem than signal reconstruction}.

Finally, it is worth comparing the model-selection performance of OST with that of the lasso for the cases when the lasso does succeed. In this regard, the most general result for model selection using the lasso states that if $X$ is close to being a tight frame in the sense that $\|X\|_2 \approx \sqrt{p/n}$ then the lasso identifies the correct model with probability exceeding $1 - O(p^{-1})$ as long as (i) the nonzero entries of $\beta$ are independent and statistically symmetric around zero, (ii) $k \precsim n/\log{p}$, and (iii) $\MAR \succsim \frac{k\log{p}}{n\cdot\SNR}$ \cite[Theorem~1.3]{candes:annstat09}. On the other hand, assume now that the design matrix $X$ has $\mu(X) \asymp n^{-1/2}$ and $\nu(X) \precsim n^{-1}$; there indeed exist design matrices that satisfy these conditions (e.g., Gaussian matrices, as proved earlier, and Alltop Gabor frames, as proved in Section~\ref{sec:gabor}). We then have from Theorem~\ref{thm:ems_cp2} (resp. Theorem~\ref{thm:sost_ems_cp2}) that OST (resp. SOST) identifies the correct model with probability exceeding $1 - O(p^{-1})$ as long as $k \precsim n/\log{p}$ and $\MAR \succsim \max\Big\{\frac{1}{\SNR}, 1\Big\} \frac{k\log{p}}{n}$. This suggests that, even for the cases in which the lasso succeeds, OST can be guaranteed to perform as well as the lasso in situations where either the energy of any nonzero entry of $\beta$ is not too far away from the average energy $(\MAR = \Theta(1))$ or the $\SNR$ is not too high $(\SNR = O(1))$. Equally importantly, and in contrast to the lasso results reported in \cite{candes:annstat09}, OST is guaranteed to attain this performance \emph{irrespective} of the values of the nonzero entries of the data vector $\beta$.

\subsubsection{Near-Optimality of OST}
We have concluded up to this point that---under certain conditions on $\MAR$ and $\SNR$---the OST algorithm can perform as well as the lasso and it performs near-optimally for Gaussian design matrices. We conclude this discussion by arguing that the OST algorithm in fact performs near-optimally for \emph{any} design matrix that satisfies $\mu(X) \asymp n^{-1/2}$ and $\nu(X) \precsim n^{-1}$ as long as $\MAR = \Theta(1)$ or $\SNR = O(1)$.\footnote{Note that it trivially follows from the Welch bound \cite{welch:tit74} that there exists no design matrix with $p \gg 1$ that satisfies $\mu(X) \asymp n^{-1/\gamma}$ with $\gamma < 2$. On the other hand, there does exist a large body of literature devoted to constructing matrices with $\mu(X) \asymp n^{-1/2}$ \cite{sarwate:seta98}.} In order to accomplish this goal, we first recall the thresholding results obtained by Donoho and Johnstone \cite{donoho:biomet94}---which form the basis of ideas such as the wavelet denoising---for the case of $p \times p$ orthonormal design matrices. Specifically, it was established in \cite{donoho:biomet94} that if $X$ is an orthonormal basis then hard thresholding the entries of $X^H y$ at $\lambda \asymp \sqrt{\sigma^2\log{p}}$ results in oracle-like performance in the sense that one recovers (with high probability) the locations of all the nonzero entries of $\beta$ that are above the noise floor.

Now the first thing to note regarding the results presented earlier in this section is the intuitively pleasing nature of the threshold proposed for the OST algorithm. Specifically, assume that $X$ is an orthonormal design and notice that, since $\mu(X) = 0$ in this case, the threshold $\lambda \asymp \max\Big\{\mu \sqrt{n \cdot \SNR}, 1\Big\}\sqrt{\sigma^2\log{p}}$ proposed earlier reduces to the threshold proposed in \cite{donoho:biomet94} \emph{and} Theorem~\ref{thm:pms_cp2} guarantees that thresholding recovers (with high probability) the locations of all the nonzero entries of $\beta$ that are above the noise floor: $\LAR_m \succsim \frac{k\log{p}}{n\cdot\SNR} \ \Rightarrow \ m \in \whcS$. Now consider instead design matrices that are not necessarily orthonormal but which satisfy $\mu(X) \asymp n^{-1/2}$ and $\nu(X) \precsim n^{-1}$. Then we have from Theorem~\ref{thm:pms_cp2} that OST identifies (with high probability) the locations of the nonzero entries of $\beta$ whose energies are greater than both the noise power and the average signal energy per nonzero entry: $\LAR_m \succsim \max\Big\{\frac{1}{\SNR}, 1\Big\} \frac{k\log{p}}{n}  \ \Rightarrow \ m \in \whcS$. In particular, under the assumption that either $\MAR = \Theta(1)$ (and since $\MAR \leq \LAR_m$) or $\SNR = O(1)$, this suggests that the OST in such situations performs in a near-optimal (oracle-like) fashion in the sense that it recovers (with high probability) the locations of all the nonzero entries of $\beta$ that are above the noise floor \emph{without} requiring the design matrix $X$ to be an orthonormal basis.

\section{Recovery of Sparse Signals Using One-Step Thresholding}\label{sec:sparse_rec}
In this section, we extend our results on model selection using OST to model-order agnostic recovery of $k$-sparse signals. In doing so, we also strengthen the results of Schnass and Vandergheynst \cite{schnass:spl07} for signal recovery using thresholding in at least three key aspects. First, we specify polynomial-time verifiable sufficient conditions under which recovery of $k$-sparse signals using OST succeeds. Second, the threshold that we specify for the OST algorithm (Algorithm~\ref{alg:OST_recon}) does not require knowledge of the model order $k$. Third, we do not impose a statistical prior on the nonzero entries of the data vector $\beta$. Note that, just like \cite{schnass:spl07}, we limit ourselves in this exposition to recovery of $k$-sparse signals in a noiseless setting; extensions of these results to noisy settings would be reported in a sequel to this paper. In other words, the measurement model that we study in this section is $y = X \beta$ and the goal is to recover the $k$-sparse $\beta$ using OST under the assumption that the true model $\cS \doteq \{i \in \{1,\dots,p\}:|\beta_i| > 0\}$ is a uniformly random $k$-subset of $\{1,\dots,p\}$.

\subsection{Main Result}
Intuitively speaking (and as noted in the discussion in Section~\ref{sec:mod_sel}), the problem of sparse-signal recovery is inherently more difficult than the problem of model selection. We capture part of this intuitive notion in the following in terms of the \emph{strong coherence property}.
\begin{definition}[The Strong Coherence Property]
An $n \times p$ design matrix $X$ having unit $\ell_2$-norm columns is said to obey the strong coherence property if the following two conditions hold:
\begin{align}
\label{eqn:scp1}
\tag{SCP-1} \mu(X) &\leq \frac{1}{60\mathrm{e}\log{p}}, \quad \text{and}\\
\label{eqn:scp2}
\tag{SCP-2} \nu(X) &\leq \frac{\mu}{\sqrt{n}}\,.
\end{align}
\end{definition}
\noindent In order to better illustrate the difference between the coherence property and the strong coherence property, note that we have from Lemma~\ref{lem:mu_gaussian} and Lemma~\ref{lem:nu_gaussian} that Gaussian design matrices satisfy the coherence property with high probability as long as $n \succsim (\log{p})^2$. On the other hand, Lemma~\ref{lem:mu_gaussian} and Lemma~\ref{lem:nu_gaussian} suggest that Gaussian design matrices satisfy the strong coherence property with high probability as long as $n \succsim (\log{p})^4$. In other words, there are scaling regimes in which Gaussian design matrices satisfy the coherence property but are not guaranteed to satisfy the strong coherence property. We are now ready to state the main result of this section that makes use of the notation developed earlier in Section~\ref{sec:mod_sel} of the paper.
\begin{theorem}[Sparse-Signal Recovery Using OST]\label{thm:sig_rec}
Suppose that the design matrix $X$ satisfies the strong coherence property and choose the threshold $\lambda = 10\mu\|y\|_2\sqrt{\frac{2\log{p}}{1-\mathrm{e}^{-1/2}}} \,$ for any $p \geq 128$. Then the OST algorithm (Algorithm~\ref{alg:OST_recon}) satisfies $\Pr(\widehat{\beta} \not= \beta) \leq 6 p^{-1}$ as long as
\begin{align}
\label{eqnthm:sig_rec}
 k \leq \min \Bigg\{\frac{p}{c_3^2 \|X\|_2^2 \log{p}}, \frac{\mu^{-2} \MAR}{c_4^2\log{p}}\Bigg\}.
\end{align}
Here, the probability of failure is only with respect to the true model $\cS$ (locations of the nonzero entries of $\beta$), while $c_3, c_4$ are positive numerical constants given by $c_3 \doteq 37\mathrm{e}$ and $c_4 \doteq 43$.
\end{theorem}

The significance of this theorem can be best put into perspective by considering the case of the design matrix $X$ being an approximately tight frame in the sense that $\|X\|_2 \approx \sqrt{p/n}$; indeed, we have that Gaussian design matrices satisfy this condition with high probability \cite{rudelson:icm10} and that Gabor frames generated from any (unit-norm) nonzero vector satisfy $\|X\|_2 \equiv \sqrt{p/n}~(= \sqrt{n})$ \cite{lawrence:jfaa05}. It then follows from Theorem~\ref{thm:sig_rec} that if $X$ satisfies the strong coherence property then OST exactly recovers any $k$-sparse vector $\beta$ with high probability as long as $k \precapprox \mu^{-2} \MAR$; in particular, if we assume that $\MAR = \Theta(1)$ then this condition reduces to $k \precapprox \mu^{-2}$. On the other hand, low-complexity sparse-recovery algorithms such as subspace pursuit \cite{dai:tit09}, CoSaMP \cite{needell:acha09}, and iterative hard thresholding \cite{blumensath:acha09} all rely on the restricted isometry property (RIP) \cite{candes:cras08}. Therefore, the guarantees provided in \cite{dai:tit09, needell:acha09, blumensath:acha09} for the case of generic design matrices are limited to $k$-sparse signals that satisfy $k \precsim \mu^{-1}$, which is much weaker than the $k \precapprox \mu^{-2}$ scaling claimed here.\footnote{Note that the $k \precsim \mu^{-1}$ claim is an easy consequence of the \emph{Ger\v{s}gorin circle theorem} \cite{gersgorin:ians31}; see, for example, \cite{devore:num98,bajwa:ciss08,tropp:acha08,haupt:tit08sub}.} We conclude this section by pointing out that if one does have knowledge of the true model order then it can be shown through a slight variation of the proof of Theorem~\ref{thm:sig_rec} that SOST (the sorted variant of the OST) can also recover sparse signals with high probability---the only difference in that case being that the constant $c_4$ in Theorem~\ref{thm:sig_rec} gets replaced with a smaller constant $c_4^\prime \doteq \sqrt{800}$.

\section{Why Gabor Frames?}\label{sec:gabor}
Our focus in Section~\ref{sec:mod_sel} and Section~\ref{sec:sparse_rec} has been on establishing that OST leads to successful model selection and sparse-signal recovery under certain conditions on three global parameters of the design matrix $X$: $\mu(X)$, $\nu(X)$, and $\|X\|_2$. As noted earlier, one particular class of design matrices that satisfies these conditions is the class of random sub-Gaussian matrices. In contrast, our focus in this section is on establishing that Gabor frames---which are collections of time- and frequency-shifts of a nonzero seed vector in $\C^n$---also tend to satisfy the aforementioned conditions on the matrix geometry. Note that Gabor frames constitute an important class of design matrices because of the facts that (i) Gabor frames are completely specified by a total of $n$ numbers that describe the seed vector, (ii) multiplications with Gabor frames (and their adjoints) can be efficiently carried out using algorithms such as the \emph{fast Fourier transform}, (iii) Gabor frames arise naturally in many important application areas such as communications, radar, and signal/image processing, and (iv) there exist deterministic constructions of Gabor frames that (as shown next) are nearly-optimal in terms of the requisite conditions on $\mu(X)$, $\nu(X)$, and $\|X\|_2$.

\subsection{Geometry of Gabor Frames and Its Implications}
A (finite) frame for $\mathbb{C}^n$ is defined as any collection of $p \geq n$ vectors that span the $n$-dimensional Hilbert space $\mathbb{C}^n$ \cite{christensen:08}. Gabor frames for $\mathbb{C}^n$ constitute an important class of frames, having applications in areas such as communications \cite{bajwa:allerton08} and radar \cite{herman:tsp09}, that are constructed from time- and frequency-shifts of a nonzero seed vector in $\C^n$. Specifically, let $g \in \mathbb{C}^n$ be a unit-norm seed vector and define $T$ to be an $n \times n$ \emph{time-shift matrix} that is generated from $g$ as follows
\begin{align}
    T(g) \doteq \begin{bmatrix}
        g_1 &  	g_n &  	&  		& 	g_2 \\
        g_2 &  	g_1 & 		\ddots &  	&  	\vdots\\
        \vdots &  	\vdots & 	\ddots 	&  	& 	g_n\\
        g_n &  	g_{n-1} &  	&  		& 	g_1
    \end{bmatrix}
\end{align}
where we write $T = T(g)$ to emphasize that $T$ is a matrix-valued function on $\mathbb{C}^n$. Next, denote the collection of $n$ samples of a discrete sinusoid with frequency $2\pi\frac{m}{n}, m \in \{0,\dots,n-1\}$ as $\omega_m \doteq \begin{bmatrix} \mathrm{e}^{j2\pi\frac{m}{n}0} & \dots & \mathrm{e}^{j2\pi\frac{m}{n}(n-1)}\end{bmatrix}^\tT$. Finally, define the corresponding $n \times n$ diagonal \emph{modulation matrices} as $W_m = \text{diag}(\omega_m)$. Then the Gabor frame generated from $g$ is an $n \times n^2$ block matrix of the form
\begin{align}
\label{eqn:gabor_def}
 X =
    \begin{bmatrix}
        W_{0} T & W_1 T & \dots & W_{n-1} T
    \end{bmatrix}.
\end{align}
In words, columns of the Gabor frame $X$ are given by downward circular shifts and modulations (frequency shifts) of the seed vector $g$. We are now ready to state the first main result concerning the geometry of Gabor frames, which follows directly from \cite{lawrence:jfaa05}.
\begin{proposition}[Spectral Norm of Gabor Frames \cite{lawrence:jfaa05}]\label{lem:gf_tight}
Gabor frames generated from nonzero (unit-norm) seed vectors are tight frames; in other words, we have that $\|X\|_2 = \sqrt{n}$\,.
\end{proposition}

Recall from Theorem~\ref{thm:sig_rec} and the subsequent discussion in Section~\ref{sec:sparse_rec} that design matrices with small spectral norms are particularly well-suited for recovery of $k$-sparse signals. In this regard, Proposition~\ref{lem:gf_tight} implies that Gabor frames are the best that one can hope for in terms of the spectral norm. The next result that we prove concerns the average coherence of Gabor frames.
\begin{theorem}[Average Coherence of Gabor Frames]\label{thm:gf_avc}
Let $X$ be a Gabor frame generated from a unit-norm seed vector $g \in \C^n$. Then, using the notation $g_{\max} \doteq \max_i |g_i|$ and $g_{\min} \doteq \min_i |g_i|$, the average coherence of $X$ can be bounded from the above as follows:
\begin{align}
	\nu(X) \leq \frac{n\,g_{\max}(\sqrt{n} - g_{\min}) + 1 - n\,g_{\min}^2}{n^2 - 1}.
\end{align}
\end{theorem}
\begin{proof}
In order to facilitate the proof of this theorem, we first map the indices of the columns of $X$ from $\{1,\dots,n^2\}$ to $\cC \doteq \{0,\dots,n-1\} \times \{0,\dots,n-1\}$ as follows
\begin{align}
 \kappa : i \mapsto \left((i \text{ mod } n)-1, \left\lfloor\frac{i-1}{n}\right\rfloor\right).
\end{align}
In words, $\kappa(i) = (\ell,m)$ signifies that the $i$-th column of $X$ corresponds to the $(\ell+1)$-th column of $W_m T$. Next, fix an index $i$ (resp. $\kappa(i) = (\ell,m)$) and make use of the above reindexing to write
\begin{align}
\label{eqn:lem_gfavc}
	\sum_{\substack{j=1\\j\not=i}}^{n^2} \langle\mathrm{x}_{\kappa(i)}, \mathrm{x}_{\kappa(j)}\rangle &= \sum_{\substack{(\ell^\prime\!,m^\prime) \in \cC\\(\ell^\prime\!,m^\prime)\not=(\ell,m)}} \langle\mathrm{x}_{\ell,m}, \mathrm{x}_{\ell^\prime\!,m^\prime}\rangle = \sum_{\substack{\ell^\prime=0\\ \ell^\prime \not= \ell}}^{n-1} \ \sum_{m^\prime = 0}^{n-1} \langle\mathrm{x}_{\ell,m}, \mathrm{x}_{\ell^\prime\!,m^\prime}\rangle + \sum_{\substack{m^\prime=0\\ m^\prime \not= m}}^{n-1} \langle\mathrm{x}_{\ell,m}, \mathrm{x}_{\ell,m^\prime}\rangle.
\end{align} 
Finally, note that we can explicitly write the columns of $X$ using \eqref{eqn:gabor_def} for any $(\ell,m) \in \cC$ as follows
\begin{align}
\label{eqn:gabor_col}
 \mathrm{x}_{\ell,m} \doteq \begin{bmatrix}
                         g_{(1 - \ell)_n} \mathrm{e}^{j2\pi\frac{m}{n}0} & \dots & g_{(n - \ell)_n} \mathrm{e}^{j2\pi\frac{m}{n}(n-1)}
                        \end{bmatrix}^\tT
\end{align}
where we use the notation $g_{(q)_n}$ as a shorthand for $g_{q \text{ mod } n}$.

The rest of the proof now follows from simple algebraic manipulations. Specifically, it is easy to see from \eqref{eqn:gabor_col} that the first term in \eqref{eqn:lem_gfavc} can be simplified as
\begin{align}
\nonumber
 	\sum_{\substack{\ell^\prime=0\\ \ell^\prime \not= \ell}}^{n-1} \ \sum_{m^\prime = 0}^{n-1} \langle\mathrm{x}_{\ell,m}, \mathrm{x}_{\ell^\prime\!,m^\prime}\rangle &= \sum_{q=1}^{n} \sum_{\substack{\ell^\prime=0\\ \ell^\prime \not= \ell}}^{n-1} g_{(q - \ell)_n}^*g_{(q - \ell^\prime)_n} \sum_{m^\prime = 0}^{n-1} \mathrm{e}^{j2\pi\frac{q-1}{n}(m^\prime - m)}\\
\nonumber
	&= \sum_{q=2}^{n} \sum_{\substack{\ell^\prime=0\\ \ell^\prime \not= \ell}}^{n-1} g_{(q - \ell)_n}^*g_{(q - \ell^\prime)_n} \sum_{m^\prime = 0}^{n-1} \mathrm{e}^{j2\pi\frac{q-1}{n}(m^\prime - m)} + \\
\label{eqn:lem_gabst2_1}
&\quad+ n \sum_{\substack{\ell^\prime=0\\ \ell^\prime \not= \ell}}^{n-1} g_{(1 - \ell)_n}^*g_{(1 - \ell^\prime)_n}
	\stackrel{(a)}{=} n\,g_{(1 - \ell)_n}^* \sum_{\substack{\ell^\prime=0\\ \ell^\prime \not= \ell}}^{n-1} g_{(1 - \ell^\prime)_n}
\end{align}
where $(a)$ in the above expression is a consequence of the fact that $\sum_{m^\prime = 0}^{n-1} \mathrm{e}^{j2\pi\frac{q-1}{n}(m^\prime - m)} = 0$ for any fixed $q \in \{2,\dots,n\}$. Likewise, we can simplify the second term in \eqref{eqn:lem_gfavc} as follows
\begin{align}
\nonumber
 	\ \sum_{\substack{m^\prime=0\\ m^\prime \not= m}}^{n-1} \langle\mathrm{x}_{\ell,m}, \mathrm{x}_{\ell,m^\prime}\rangle &= \sum_{q=1}^{n}  g_{(q - \ell)_n}^*g_{(q - \ell)_n} \sum_{\substack{m^\prime=0\\ m^\prime \not= m}}^{n-1} \mathrm{e}^{j2\pi\frac{q-1}{n}(m^\prime - m)}\\
\nonumber
	&= \sum_{q=2}^{n}  \big|g_{(q - \ell)_n}\!\big|^2 \sum_{\substack{m^\prime=0\\ m^\prime \not= m}}^{n-1} \mathrm{e}^{j2\pi\frac{q-1}{n}(m^\prime - m)} + \big|g_{(1 - \ell)_n}\!\big|^2 \sum_{\substack{m^\prime=0\\ m^\prime \not= m}}^{n-1} 1\\
\label{eqn:lem_gabst2_2}
	&\stackrel{(b)}{=} - \sum_{q=2}^{n} \big|g_{(q - \ell)_n}\!\big|^2 + (n-1) \big|g_{(1 - \ell)_n}\!\big|^2 = - 1 + n\big|g_{(1 - \ell)_n}\!\big|^2 
\end{align}
where $(b)$ follows from the fact that $\sum_{m^\prime \not= m} \mathrm{e}^{j2\pi\frac{q-1}{n}(m^\prime - m)} = -1$ for any fixed $q \in \{2,\dots,n\}$.

To conclude the theorem, note from \eqref{eqn:lem_gfavc}, \eqref{eqn:lem_gabst2_1}, and \eqref{eqn:lem_gabst2_2} that we can write
\begin{align}
\nonumber
 \max_{i \in \{1,\dots,n^2\}} \bigg|\sum_{\substack{j=1\\j\not=i}}^{n^2} \langle\mathrm{x}_i, \mathrm{x}_j\rangle\bigg| &= \max_\ell \bigg|n\,g_{(1 - \ell)_n}^* \sum_{\substack{\ell^\prime=0\\ \ell^\prime \not= \ell}}^{n-1} g_{(1 - \ell^\prime)_n}  - 1 + n\big|g_{(1 - \ell)_n}\!\big|^2\bigg|\\
\nonumber
	&\stackrel{(c)}{\leq} \max_{r \in \{1,\dots,n\}} \bigg|n\,g_r^* \sum_{\substack{s=1\\ s \not= r}}^{n} g_s\bigg| + \max_{r \in \{1,\dots,n\}} \Big|- 1 + n|g_r|^2\Big|\\
\nonumber
	&\leq n \max_{r \in \{1,\dots,n\}} |g_r| \sum_{\substack{s=1\\ s \not= r}}^{n} |g_s| + 1 - n\,g_{\min}^2\\
\label{eqn:lem_gf_fin}
	&\stackrel{(d)}{\leq} n\,g_{\max}(\sqrt{n} - g_{\min}) + 1 - n\,g_{\min}^2.
\end{align}
Here, $(c)$ mainly follows from the triangle inequality and a simple reindexing argument, while $(d)$ mainly follows from the Cauchy--Schwarz inequality since $\sum_{\substack{s=1\\ s \not= r}}^{n} |g_s| = \|g\|_1 - |g_r| \leq \sqrt{n} - g_{\min}$. The proof of the theorem now follows by dividing the above expression by $n^2 -1$.
\end{proof}

In words, Theorem~\ref{thm:gf_avc} states that the average coherence of Gabor frames cannot be too large. In particular, it implies that Gabor frames generated from unimodal (unit-norm) seed vectors (i.e., seed vectors characterized by $g_{\min} \asymp g_{\max} \asymp n^{-1/2}$) satisfy $\nu(X) \precsim n^{-1}$. On the other hand, recall that the Welch bound \cite{welch:tit74} dictates that $\mu(X) \geq (n+1)^{-1/2}$ for Gabor frames. It is therefore easy to conclude from these two facts that Gabor frames generated from unimodal seed vectors are automatically guaranteed to satisfy the coherence property (resp. strong coherence property) as long as $\mu(X) \precsim (\log{p})^{-1/2}$ (resp. $\mu(X) \precsim (\log{p})^{-1}$). In the context of model selection and sparse-signal recovery, Theorem~\ref{thm:gf_avc} therefore suggests that Gabor frames generated from unimodal seed vectors are the best that one can hope for in terms of the average coherence.

Finally, recall from the discussions in Section~\ref{sec:mod_sel} and Section~\ref{sec:sparse_rec} that---among the class of matrices that satisfy the (strong) coherence property---design matrices with small worst-case coherence are particularly well-suited for model selection and sparse-signal recovery. In the context of Gabor frames, the goal then is to design unimodal seed vectors that yield Gabor frames with the smallest-possible worst-case coherence. This, however, is an active area of mathematical research and a number of researchers have looked at this problem in recent years; see, e.g., \cite{strohmer:acha03}. As such, we can simply leverage some of the existing research in this area in order to provide explicit constructions of Gabor frames that satisfy the (strong) coherence property with nearly-optimal worst-case coherence.

Specifically, let $n \geq 5$ be a prime number and construct a unimodal seed vector $g \in \mathbb{C}^n$ as follows
\begin{align}
\label{eqn:alltopvec}
	g = \begin{bmatrix}
	    	\frac{1}{\sqrt{n}}\mathrm{e}^{j2\pi\frac{0^3}{n}} & \frac{1}{\sqrt{n}}\mathrm{e}^{j2\pi\frac{1^3}{n}} & \dots & \frac{1}{\sqrt{n}}\mathrm{e}^{j2\pi\frac{(n-1)^3}{n}} 
	    \end{bmatrix}^\tT.
\end{align}
The sequence $\left\{\frac{1}{\sqrt{n}}\mathrm{e}^{j2\pi\frac{q^3}{n}}\right\}_{q=0}^{n-1}$ is termed as the \emph{Alltop sequence} \cite{alltop:tit80} in the literature. This sequence has the property that its autocorrelation decays very fast and, therefore, it is particularly well-suited for generating Gabor frames with small worst-case coherence. In particular, it was established recently in \cite{strohmer:acha03} that Gabor frames generated from the Alltop seed vector $g$ given in \eqref{eqn:alltopvec} satisfy
\begin{align}
	\mu(X) \doteq \max\limits_{i,j:i \neq j} \big|\langle\mathrm{x}_i, \mathrm{x}_j\rangle\big| \leq \frac{1}{\sqrt{n}}.
\end{align}
In addition, since we have that $g_{\min} = g_{\max} = n^{-1/2}$ for the Alltop seed vector, it is easy to check using Theorem~\ref{thm:gf_avc} that the average coherence of Alltop Gabor frames satisfies $\nu(X) \leq (n+1)^{-1} \leq \mu(X)/\sqrt{n}$. An immediate consequence of this discussion is that all the results reported in Section~\ref{sec:mod_sel} and Section~\ref{sec:sparse_rec} in the context of model selection and sparse-signal recovery using OST apply directly to the case of Alltop Gabor frames. In particular, it follows from Theorem~\ref{thm:sig_rec} that Alltop Gabor frames together with OST are guaranteed to recover most $k$-sparse signals---regardless of the statistical dependence across the nonzero entries of $\beta$---as long as $k \precapprox n$ and $\MAR = \Theta(1)$. In contrast, the only other results available in the sparse-signal recovery literature for Alltop Gabor frames are based on the higher-complexity basis pursuit \cite{donoho:siamjsc98} and require the nonzero entries of $\beta$ to be independent and statistically symmetric around zero for the case when $\sqrt{n} \precsim k \precapprox n$ \cite{pfander:tsp08,herman:tsp09}.

\section{Proofs of Main Results}\label{sec:proofs}
In this section, we provide detailed proofs of the main results reported in Section~\ref{sec:mod_sel} and Section~\ref{sec:sparse_rec}. Before proceeding further, however, it is advantageous to develop some notation that will facilitate our forthcoming analysis. In this regard, recall that the true model $\cS$ is taken to be a uniformly random $k$-subset of $\nN{p} \doteq \{1,\dots,p\}$. We can therefore write the data vector $\beta$ under this assumption as concatenation of a random permutation matrix and a deterministic $k$-sparse vector. Specifically, let $\bar{z} \in \C^p$ be a \emph{deterministic} $k$-sparse vector that we write (without loss of generality) as
\begin{align}
 \bar{z} \doteq \Big(\underbrace{z_1,\dots,z_k}_{\doteq \, z \in \C^k}, \underbrace{0,\dots,0}_{(p-k) \text{ times}}\Big)^\tT
\end{align}
and let $P_\pi$ be a $p \times p$ random permutation matrix; in other words,
\begin{align}
 P_\pi \doteq \begin{bmatrix}
               \textrm{e}_{\pi_1} & \textrm{e}_{\pi_2} & \dots & \textrm{e}_{\pi_p}
              \end{bmatrix}^\tT
\end{align}
where $\textrm{e}_j$ denotes the $j$-th column of the canonical basis $I$ and $\bar{\Pi} \doteq (\pi_1,\dots,\pi_p)$ is a random permutation of $\nN{p}$. Then the assumption that the model $\cS$ is a random subset of $\nN{p}$ is equivalent to stating that the data vector $\beta$ can be written as $\beta = P_\pi \bar{z}$. In other words, the measurement vector $y$ can be expressed as
\begin{align}
 y = X \beta + \eta = X P_\pi \bar{z} + \eta = X_\Pi z + \eta
\end{align}
where $\Pi \doteq (\pi_1,\dots,\pi_k)$ denotes the first $k$ elements of the random permutation $\bar{\Pi}$, $X_\Pi$ denotes the $n \times k$ submatrix obtained by collecting the columns of $X$ corresponding to the indices in $\Pi$, and the vector $z \in \C^k$ represents the $k$ nonzero entries of $\beta$.

\subsubsection{Proof of Theorem~\ref{thm:ems_cp}}
The general road map for the proof of Theorem~\ref{thm:ems_cp} is as follows. Below, we first introduce the notion of $(k,\epsilon,\delta)$-\emph{statistical orthogonality condition} (StOC). We next establish the relationship between the StOC parameters and the worst-case and average coherence of $X$ in Lemma~\ref{lem:stoc1} and Lemma~\ref{lem:stoc2}. We then provide a proof of Theorem~\ref{thm:ems_cp} by first showing that if $X$ satisfies the StOC then OST recovers $\cS$ with high probability and then relating the results of Lemma~\ref{lem:stoc1} and Lemma~\ref{lem:stoc2} to the coherence property.
\begin{definition}[$(k,\epsilon,\delta)$-Statistical Orthogonality Condition]
Let $\bar{\Pi} = (\pi_1,\dots,\pi_p)$ be a random permutation of $\nN{p}$, and define $\Pi \doteq (\pi_1,\dots,\pi_k)$ and $\Pi^c \doteq (\pi_{k+1},\dots,\pi_p)$ for any $k \in \nN{p}$. Then the $n \times p$ (normalized) design matrix $X$ is said to satisfy the $(k,\epsilon,\delta)$-statistical orthogonality condition if there exist $\epsilon, \delta \in [0,1)$ such that the inequalities
\begin{align}
\label{eqn:stoc1}
\tag{StOC-1} \|(X_\Pi^\tH X_\Pi - I)z\|_\infty &\leq \epsilon \|z\|_2\\
\label{eqn:stoc2}
\tag{StOC-2} \|X_{\Pi^c}^\tH X_\Pi z\|_\infty &\leq \epsilon \|z\|_2
\end{align}
hold for every \emph{fixed} $z \in \C^k$ with probability exceeding $1-\delta$ (with respect to the random permutation $\bar{\Pi}$).
\end{definition}
\begin{remark}
Note that the StOC derives its name from the fact that if $X$ is a $p \times p$ orthonormal matrix then it trivially satisfies the StOC for every $k \in \nN{p}$ with $\epsilon = \delta = 0$. In addition, although we will not use this fact explicitly in the paper, it can be checked that if $X$ satisfies $(k,\epsilon,\delta)$-StoC then it \emph{approximately} preserves the $\ell_2$-norms of $k$-sparse signals with probability exceeding $1-\delta$ as long as $k < \epsilon^{-2}$.
\end{remark}
Having defined StOC, our goal in the next two lemmas is to relate the StOC parameters $k, \epsilon$, and $\delta$ to the worst-case and average coherence of the design matrix $X$.
\begin{lemma}\label{lem:stoc1}
Let $\Pi = (\pi_1,\dots,\pi_k)$ denote the first $k$ elements of a random permutation of $\nN{p}$ and choose a parameter $a \geq 1$. Then, for any fixed $z \in \C^k$, $\epsilon \in [0,1)$, and $k \leq \min\big\{\epsilon^2\nu^{-2},(1+a)^{-1}p\big\}$, we have 
\begin{align}
 \Pr\Big(\big\{\text{$X$ \emph{does not satisfy} \eqref{eqn:stoc1}}\big\}\Big) \leq 4k\exp\bigg(-\frac{(\epsilon-\sqrt{k}\,\nu)^2}{16(2+ a^{-1})^2\mu^2}\bigg).
\end{align}
\end{lemma}
\begin{proof}
The proof of this lemma relies heavily on the so-called \emph{method of bounded differences} (MOBD) \cite{mcdiarmid:89}. Specifically, we begin by noting that $\big\|(X_\Pi^\tH X_\Pi - I)z\big\|_\infty = \max\limits_i \bigg|\sum\limits_{j\not=i} z_j \langle\mathrm{x}_{\pi_i},\mathrm{x}_{\pi_j}\rangle\bigg|$. Therefore for a fixed index $i$, and conditioned on the event $\cA_{i^\prime} \doteq \big\{\pi_i = i^\prime\big\}$, we have the following equality from basic probability theory
\begin{align}
\label{eqn:lem_stoc1_prob}
 \Pr\bigg(\big|\sum\limits_{\substack{j=1\\j\not=i}}^k z_j \langle\mathrm{x}_{\pi_i},\mathrm{x}_{\pi_j}\rangle\big| > \epsilon\|z\|_2\bigg|\cA_{i^\prime}\bigg) = \Pr\bigg(\big|\sum\limits_{\substack{j=1\\j\not=i}}^k z_j \langle\mathrm{x}_{i^\prime},\mathrm{x}_{\pi_j}\rangle\big| > \epsilon\|z\|_2\bigg|\cA_{i^\prime}\bigg).
\end{align}

Next, in order to apply the MOBD to obtain an upper bound for \eqref{eqn:lem_stoc1_prob}, we first define a random $(k-1)$-tuple $\Pi^{-i} \doteq (\pi_1,\dots,\pi_{i-1},\pi_{i+1},\dots,\pi_k)$ and then construct a Doob martingale $(M_0,M_1,\dots,M_{k-1})$ as follows:
\begin{align}
 M_0 = \E\Big[\sum\limits_{\substack{j=1\\j\not=i}}^k z_j \langle\mathrm{x}_{i^\prime},\mathrm{x}_{\pi_j}\rangle\Big|\cA_{i^\prime}\Big] \quad \text{and} \quad 
 M_\ell = \E\Big[\sum\limits_{\substack{j=1\\j\not=i}}^k z_j \langle\mathrm{x}_{i^\prime},\mathrm{x}_{\pi_j}\rangle\Big|\pi^{-i}_{1\rightarrow\ell},\cA_{i^\prime}\Big], \ \ell=1,\dots,k-1
\end{align}
where $\pi^{-i}_{1\rightarrow\ell}$ denotes the first $\ell$ elements of $\Pi^{-i}$. The first thing to note here is that we have from the linearity of (conditional) expectation
\begin{align}
\nonumber
 \big|M_0\big| &= \Big|\sum_{j\not=i} z_j \E\big[\langle\mathrm{x}_{i^\prime},\mathrm{x}_{\pi_j}\rangle|\cA_{i^\prime}\big]\Big| \leq \sum_{j\not=i} \big|z_j\big| \Big|\E\big[\langle\mathrm{x}_{i^\prime},\mathrm{x}_{\pi_j}\rangle|\cA_{i^\prime}\big]\Big| \displaybreak[0]\\
\label{eqn:lem_stoc1_M0}
&\stackrel{(a)}{\leq} \sum_{j\not=i} \big|z_j\big| \Bigg|\sum\limits_{\substack{q=1\\q\not=i^\prime}}^p\frac{1}{p-1} \langle\mathrm{x}_{i^\prime},\mathrm{x}_{q}\rangle\Bigg| \stackrel{(b)}{\leq} \nu\,\|z\|_1 \leq \sqrt{k}\,\nu\,\|z\|_2
\end{align}
where $(a)$ follows from the fact that, conditioned on $\cA_{i^\prime}$, $\pi_j$ has a uniform distribution over $\nN{p} - \{i^\prime\}$, while $(b)$ is mainly a consequence of the definition of average coherence. In addition, if we use $\pi^{-i}_\ell$ to denote the $\ell$-th element of $\Pi^{-i}$ and define
\begin{align}
	M_\ell(r) \doteq \E\Big[\sum\limits_{\substack{j=1\\j\not=i}}^k z_j 	\langle\mathrm{x}_{i^\prime},\mathrm{x}_{\pi_j}\rangle\Big|\pi^{-i}_{1\rightarrow\ell-1},\pi^{-i}_\ell=r,\cA_{i^\prime}\Big], \ \ell=1,\dots,k-1
\end{align}
then, since $(M_0,M_1,\dots,M_{k-1})$ is a Doob martingale, it can be easily verified that $|M_\ell - M_{\ell-1}|$ is upper bounded by $\sup_{r,s} \big[M_\ell(r) - M_\ell(s)\big]$ (see, e.g., \cite{motwani:95}).

Now in order to obtain an upper bound for $\sup_{r,s} \big[M_\ell(r) - M_\ell(s)\big]$, notice that
\begin{align}
\nonumber
	\Big|M_\ell(r) - M_\ell(s)\Big| &= \Bigg|\sum_{j\not=i} z_j 
\bigg(\E\Big[\langle\mathrm{x}_{i^\prime},\mathrm{x}_{\pi_j}\rangle\Big|\pi^{-i}_{1\rightarrow\ell-1},\pi^{-i}_\ell=r,\cA_{i^\prime}\Big] - \E\Big[\langle\mathrm{x}_{i^\prime},\mathrm{x}_{\pi_j}\rangle\Big|\pi^{-i}_{1\rightarrow\ell-1},\pi^{-i}_\ell=s,\cA_{i^\prime}\Big]\bigg)\Bigg|\\
\nonumber
&\leq \sum_{j\not=i} \big|z_j\big| \bigg|\underbrace{\E\Big[\langle\mathrm{x}_{i^\prime},\mathrm{x}_{\pi_j}\rangle\Big|\pi^{-i}_{1\rightarrow\ell-1},\pi^{-i}_\ell=r,\cA_{i^\prime}\Big] - \E\Big[\langle\mathrm{x}_{i^\prime},\mathrm{x}_{\pi_j}\rangle\Big|\pi^{-i}_{1\rightarrow\ell-1},\pi^{-i}_\ell=s,\cA_{i^\prime}\Big]}_{\doteq\, d_{\ell,j}}\bigg|\\
&= \sum_{\substack{j \leq\,\ell+1\\j\not=i}}\big|z_j\big|\big|d_{\ell,j}\big| + \sum_{\substack{j >\,\ell+1\\j\not=i}}\big|z_j\big|\big|d_{\ell,j}\big|.
\end{align}
In addition, we have that for every $j > \ell+1, j \not= i$, the random variable $\pi_j$ has a uniform distribution over $\nN{p} - \{\pi^{-i}_{1\rightarrow\ell-1}, r, i^\prime\}$ when conditioned on $\{\pi^{-i}_{1\rightarrow\ell-1}, \pi^{-i}_\ell = r, i^\prime\}$, whereas $\pi_j$ has a uniform distribution over $\nN{p} - \{\pi^{-i}_{1\rightarrow\ell-1}, s, i^\prime\}$ when conditioned on $\{\pi^{-i}_{1\rightarrow\ell-1}, \pi^{-i}_\ell = s, i^\prime\}$. Therefore, we obtain
\begin{align}
	|d_{\ell,j}| = \frac{1}{p-\ell-1}\Big|\langle\mathrm{x}_{i^\prime},\mathrm{x}_{r}\rangle - \langle\mathrm{x}_{i^\prime},\mathrm{x}_{s}\rangle\Big| \leq \frac{2 \mu}{p-\ell-1}, \ \forall~j > \ell+1, j \not= i.
\end{align}
Similarly, it can be shown that $\sum_{\substack{j \leq\,\ell+1\\j\not=i}}\big|z_j\big|\big|d_{\ell,j}\big| \leq \big|z_{\ell+1}\big| 2 \mu$ when $i \leq \ell$, $\sum_{\substack{j \leq\,\ell+1\\j\not=i}}\big|z_j\big|\big|d_{\ell,j}\big| \leq \big|z_{\ell}\big| 2 \mu$ when $i = \ell+1$, and $\sum_{\substack{j \leq\,\ell+1\\j\not=i}}\big|z_j\big|\big|d_{\ell,j}\big| \leq (|z_{\ell}| + \frac{|z_{\ell+1}|}{p-\ell-1})2 \mu$ when $i > \ell+1$. Consequently, regardless of the initial choice of $i$, we conclude that
\begin{align}
\label{eqn:lem_stoc1_d_ell}
 \sup_{r,s} \big[M_\ell(r) - M_\ell(s)\big] \leq 2\mu\Big(\underbrace{|z_\ell| + |z_{\ell+1}| + \frac{1}{p-\ell-1}\sum_{j >\,\ell+1}|z_j|}_{\doteq\,d_\ell}\Big).
\end{align}

We have now established that $(M_0,M_1,\dots,M_{k-1})$ is a (real- or complex-valued) bounded-difference martingale sequence with $|M_\ell - M_{\ell-1}| \leq 2 \mu d_\ell$ for $\ell=1,\dots,k-1$. Therefore, under the assumption that $k \leq \epsilon^2\nu^{-2}$ and since it has been established in \eqref{eqn:lem_stoc1_M0} that $|M_0| \leq \sqrt{k}\,\nu\,\|z\|_2$, it is easy to see that
\begin{align}
\nonumber
 \Pr\bigg(\big|\sum\limits_{\substack{j=1\\j\not=i}}^k z_j \langle\mathrm{x}_{i^\prime},\mathrm{x}_{\pi_j}\rangle\big| > \epsilon\|z\|_2\bigg|\cA_{i^\prime}\bigg) &\leq \Pr\bigg(\big|M_{k-1} - M_0| > \epsilon\|z\|_2 - \sqrt{k}\,\nu\,\|z\|_2\bigg|\cA_{i^\prime}\bigg)\\
\label{eqn:lem_stoc1_finalbd}
&\stackrel{(c)}{\leq} 4\exp\Bigg(-\frac{(\epsilon-\sqrt{k}\,\nu)^2\|z\|_2^2}{16\mu^2\sum\limits_{\ell=1}^{k-1}d_\ell^2}\Bigg)
\end{align}
where $(c)$ follows from the complex Azuma inequality for bounded-difference martingale sequences (see Lemma~\ref{lem:az_ineq} in Appendix~\ref{app:conc_ineq}). Further, it can be established through routine calculations from \eqref{eqn:lem_stoc1_d_ell} that $\sum_{\ell=1}^{k-1}d_\ell^2 \leq (2 + a^{-1})^2 \|z\|_2^2$ since $k \leq p/(1+a)$. Combining all these facts together, we finally obtain
\begin{align}
\nonumber
 \Pr\bigg(\big\|(X_\Pi^\tH X_\Pi - I)z\big\|_\infty > \epsilon\|z\|_2\bigg) &\stackrel{(d)}{\leq} k\,\Pr\bigg(\big|\sum\limits_{\substack{j=1\\j\not=i}}^k z_j \langle\mathrm{x}_{\pi_i},\mathrm{x}_{\pi_j}\rangle\big| > \epsilon\|z\|_2\bigg)\\
\nonumber
	&= k\,\sum_{i^\prime=1}^{p} \Pr\bigg(\big|\sum\limits_{\substack{j=1\\j\not=i}}^k z_j \langle\mathrm{x}_{i^\prime},\mathrm{x}_{\pi_j}\rangle\big| > \epsilon\|z\|_2\bigg|\cA_{i^\prime}\bigg) \Pr\left(\cA_{i^\prime}\right)\\
	&\stackrel{(e)}{\leq} 4k\exp\left(-\frac{(\epsilon-\sqrt{k}\,\nu)^2}{16(2+ a^{-1})^2\mu^2}\right)
\end{align}
where $(d)$ follows from the union bound and the fact that the $\pi_i$'s are identically (though not independently) distributed, while $(e)$ follows from \eqref{eqn:lem_stoc1_finalbd} and the fact that $\pi_i$ has a uniform distribution over $\nN{p}$.
\end{proof}
\begin{lemma}\label{lem:stoc2}
Let $\Pi = (\pi_1,\dots,\pi_k)$ and $\Pi^c = (\pi_{k+1},\dots,\pi_p)$ denote the first $k$ and the last $(p-k)$ elements of a random permutation of $\nN{p}$, respectively, and choose a parameter $a \geq 1$. Then, for any fixed $z \in \C^k$, $\epsilon \in [0,1)$, and $k \leq \min\big\{\epsilon^2\nu^{-2},(1+a)^{-1}p\big\}$, we have 
\begin{align}
 \Pr\Big(\big\{\text{$X$ \emph{does not satisfy} \eqref{eqn:stoc2}}\big\}\Big) \leq 4(p-k)\exp\bigg(-\frac{(\epsilon-\sqrt{k}\,\nu)^2}{8(1+a^{-1})^2\mu^2}\bigg).
\end{align} 
\end{lemma}
\begin{proof}
The proof of this lemma is very similar to that of Lemma~\ref{lem:stoc1} and also relies on the MOBD. To begin with, we note that $\big\|X_{\Pi^c}^\tH X_\Pi z\big\|_\infty = \max\limits_{i\in\nN{p-k}} \bigg|\sum\limits_{j} z_j \langle\mathrm{x}_{\pi^c_i},\mathrm{x}_{\pi_j}\rangle\bigg|$, where $\nN{p-k} \doteq \{1,\dots,p-k\}$ and $\pi^c_i$ denotes the $i$-th element of $\Pi^c$. Then for a fixed index $i \in \nN{p-k}$, and conditioned on the event $\cA_{i^\prime} \doteq \{\pi^c_i = i^\prime\}$, we again have the following equality
\begin{align}
\label{eqn:lem_stoc2_prob}
 \Pr\bigg(\big|\sum\limits_{j=1}^k z_j \langle\mathrm{x}_{\pi^c_i},\mathrm{x}_{\pi_j}\rangle\big| > \epsilon\|z\|_2\bigg|\cA_{i^\prime}\bigg) = \Pr\bigg(\big|\sum\limits_{j=1}^k z_j \langle\mathrm{x}_{i^\prime},\mathrm{x}_{\pi_j}\rangle\big| > \epsilon\|z\|_2\bigg|\cA_{i^\prime}\bigg).
\end{align}
Next, as in the case of Lemma~\ref{lem:stoc1}, we construct a Doob martingale sequence $(M_0,M_1,\dots,M_k)$ as follows:
\begin{align}
 M_0 = \E\Big[\sum\limits_{j=1}^k z_j \langle\mathrm{x}_{i^\prime},\mathrm{x}_{\pi_j}\rangle\Big|\cA_{i^\prime}\Big] \quad \text{and} \quad 
 M_\ell = \E\Big[\sum\limits_{j=1}^k z_j \langle\mathrm{x}_{i^\prime},\mathrm{x}_{\pi_j}\rangle\Big|\pi_{1\rightarrow\ell},\cA_{i^\prime}\Big], \ \ell=1,\dots,k
\end{align}
where $\pi_{1\rightarrow\ell}$ now denotes the first $\ell$ elements of $\Pi$. Then, since $\pi_j$ has a uniform distribution over $\nN{p} - \{i^\prime\}$ when conditioned on $\cA_{i^\prime}$, we once again have the bound $\big|M_0\big| \leq \sqrt{k}\,\nu\,\|z\|_2$. Therefore, the only remaining thing that we need to show in order to be able to apply the complex Azuma inequality to the constructed martingale $(M_0,M_1,\dots,M_k)$ is that $|M_\ell - M_{\ell-1}|$ is suitably bounded.

In this regard, we once again define $M_\ell(r) \doteq \E\Big[\sum\limits_{j=1}^k z_j \langle\mathrm{x}_{i^\prime},\mathrm{x}_{\pi_j}\rangle\Big|\pi_{1\rightarrow\ell-1}, \pi_\ell=r,\cA_{i^\prime}\Big]$ and note that
\begin{align}
\nonumber
	\Big|M_\ell(r) - M_\ell(s)\Big| &= \Bigg|\sum_{j} z_j 
\bigg(\E\Big[\langle\mathrm{x}_{i^\prime},\mathrm{x}_{\pi_j}\rangle\Big|\pi_{1\rightarrow\ell-1},\pi_\ell=r,\cA_{i^\prime}\Big] - \E\Big[\langle\mathrm{x}_{i^\prime},\mathrm{x}_{\pi_j}\rangle\Big|\pi_{1\rightarrow\ell-1},\pi_\ell=s,\cA_{i^\prime}\Big]\bigg)\Bigg|\\
&\leq \big|z_\ell\big|\Big|\langle\mathrm{x}_{i^\prime},\mathrm{x}_{r}\rangle - \langle\mathrm{x}_{i^\prime},\mathrm{x}_{s}\rangle\Big| + \frac{\Big|\langle\mathrm{x}_{i^\prime},\mathrm{x}_{r}\rangle - \langle\mathrm{x}_{i^\prime},\mathrm{x}_{s}\rangle\Big|}{p-\ell-1}\sum_{j > \ell}\big|z_j\big| \leq 2\mu\Big(\underbrace{|z_\ell| + \frac{\sum_{j > \ell}\big|z_j\big|}{p-\ell-1}}_{\doteq\,d_\ell}\Big)
\end{align}
which implies that $\sup_{r,s} \big[M_\ell(r) - M_\ell(s)\big] \leq 2 \mu d_\ell, \ \ell=1,\dots,k$. Consequently, we have now established that $(M_0,M_1,\dots,M_k)$ is a bounded-difference martingale with $|M_\ell - M_{\ell-1}| \leq 2 \mu d_\ell$. Therefore, since $k \leq \epsilon^2\nu^{-2}$ and $|M_0| \leq \sqrt{k}\,\nu\,\|z\|_2$, we once again have from the complex Azuma inequality that
\begin{align}
\nonumber
 \Pr\bigg(\big|\sum\limits_{j=1}^k z_j \langle\mathrm{x}_{i^\prime},\mathrm{x}_{\pi_j}\rangle\big| > \epsilon\|z\|_2\bigg|\cA_{i^\prime}\bigg) &\leq \Pr\bigg(\big|M_k - M_0| > \epsilon\|z\|_2 - \sqrt{k}\,\nu\,\|z\|_2\bigg|\cA_{i^\prime}\bigg)\\
\label{eqn:lem_stoc2_finalbd}
&\stackrel{(a)}{\leq} 4\exp\left(-\frac{(\epsilon-\sqrt{k}\,\nu)^2}{8(1+a^{-1})^2\mu^2}\right)
\end{align}
where $(a)$ follows from the fact that $\sum_{\ell=1}^{k}d_\ell^2 \leq (1+a^{-1})^2\|z\|_2^2$ since $k \leq p/(1+a)$. Combining all these facts together, we finally obtain the claimed result as follows
\begin{align}
\nonumber
 \Pr\bigg(\big\|X_{\Pi^c}^\tH X_\Pi z\big\|_\infty > \epsilon\|z\|_2\bigg) &\stackrel{(b)}{\leq} (p-k)\,\Pr\bigg(\big|\sum\limits_{j=1}^k z_j \langle\mathrm{x}_{\pi^c_i},\mathrm{x}_{\pi_j}\rangle\big| > \epsilon\|z\|_2\bigg)\\
 \nonumber
	&\leq (p-k)\,\sum_{i^\prime=1}^p \Pr\bigg(\big|\sum\limits_{j=1}^k z_j \langle\mathrm{x}_{i^\prime},\mathrm{x}_{\pi_j}\rangle\big| > \epsilon\|z\|_2\bigg|\cA_{i^\prime}\bigg) \Pr\left(\cA_{i^\prime}\right)\\
	&\stackrel{(c)}{\leq} 4(p-k)\exp\left(-\frac{(\epsilon-\sqrt{k}\,\nu)^2}{8(1+a^{-1})^2\mu^2}\right)
\end{align}
where $(b)$ follows from the union bound and the fact that the $\pi^c_i$'s are identically (though not independently) distributed, while $(c)$ follows from \eqref{eqn:lem_stoc2_finalbd} and the fact that $\pi^c_i$ has a uniform distribution over $\nN{p}$.
\end{proof}

Note that Lemma~\ref{lem:stoc1} and Lemma~\ref{lem:stoc2} collectively prove through a simple union bound argument that an $n \times p$ design matrix $X$ satisfies $(k,\epsilon,\delta)$-StOC for any $\epsilon \in [0,1)$ with $\delta \leq 4p\exp\left(-\frac{(\epsilon-\sqrt{k}\,\nu)^2}{16(2+ a^{-1})^2\mu^2}\right)$ for any $a \geq 1$ as long as $k \leq \min\big\{\epsilon^2\nu^{-2},(1+a)^{-1}p\big\}$. We are now ready to provide a proof of Theorem~\ref{thm:ems_cp}.
\begin{proof}[Proof (Theorem~\ref{thm:ems_cp})]
We begin by making use of the notation developed at the start of this section and writing the signal proxy $f = X^\tH y$ as $f = X^\tH X_\Pi z + X^\tH \eta$. Now, let $\Pi^c = (\pi_{k+1},\dots,\pi_p)$ denote the last $(p-k)$ elements of $\bar{\Pi}$ and note that we need to show that $\|f_{\Pi^c}\|_\infty \leq \lambda$ and $\min\limits_{i\in\{1,\dots,k\}}|f_{\pi_i}| > \lambda$ in order to establish that $\whcS = \cS$.

In this regard, we first assume that $X$ satisfies $(k,\epsilon,\delta)$-StOC and define $\lambda_\epsilon \doteq \max\Big\{\frac{1}{t}\epsilon\|z\|_2,\frac{1}{1-t}2\sqrt{\sigma^2\log{p}}\Big\}$ for any $t \in (0,1)$. Next, it can be verified through Lemma~\ref{lem:gaussian} in Appendix~\ref{app:conc_ineq} that $\teta \doteq X^H\eta$ satisfies $\|\teta\|_\infty \leq 2\sqrt{\sigma^2\log{p}}$ with probability exceeding $1 - 2(\sqrt{2\pi\log{p}} \cdot p)^{-1}$. Now define the event
\begin{align}
\label{eqn:thm_ems_G}
  \cG \doteq \bigg\{\Big\{\text{$X$ satisfies \eqref{eqn:stoc1} and \eqref{eqn:stoc2}}\Big\}\bigcap\Big\{\|\teta\|_\infty \leq 2\sqrt{\sigma^2\log{p}}\Big\}\bigg\}
\end{align}
and notice that we trivially have $\Pr(\cG) > 1 - \delta - 2(\sqrt{2\pi\log{p}} \cdot p)^{-1}$. Further, conditioned on the event $\cG$, we have
\begin{align}
\nonumber
 \|f_{\Pi^c}\|_\infty &\stackrel{(a)}{\leq} \|X_{\Pi^c}^\tH X_\Pi z\|_\infty + \|X_{\Pi^c}^\tH \eta\|_\infty\\
\label{eqn:thm_ems_Scomp}
 &\stackrel{(b)}{\leq} \epsilon\|z\|_2 + 2\sqrt{\sigma^2\log{p}} \ \stackrel{(c)}{\leq} \lambda_\epsilon
\end{align}
where $(a)$ follows from the triangle inequality, $(b)$ is mainly a consequence of the conditioning on the event $\cG$, and $(c)$ follows from the definition of $\lambda_\epsilon$. Next, we define $r = (X_\Pi^\tH X_\Pi - I)z$ and notice that, conditioned on the event $\cG$, we have for any $i \in \nN{k} \doteq \{1,\dots,k\}$ the following inequality:
\begin{align}
\nonumber
 |f_{\pi_i}| &= |z_i + r_i + \teta_{\pi_i}| \geq |z_i| - \|r\|_\infty - \|\teta\|_\infty\\
\label{eqn:thm_ems_S}
 &\stackrel{(d)}{\geq} \beta_{\min} - \epsilon\|z\|_2 - 2\sqrt{\sigma^2\log{p}} \ \stackrel{(e)}{\geq} \beta_{\min} - \lambda_\epsilon\,.
\end{align}
Here, $(d)$ follows from the conditioning on $\cG$, while $(e)$ is a simple consequence of the choice of $\lambda_\epsilon$. It can therefore be concluded from \eqref{eqn:thm_ems_Scomp} and \eqref{eqn:thm_ems_S} that if $X$ satisfies $(k,\epsilon,\delta)$-StOC and the OST algorithm uses the threshold $\lambda_\epsilon$ then $\Pr(\whcS \not= \cS) \leq \Pr(\cG^c)$ as long as $\beta_{\min} > 2\lambda_\epsilon$.

Finally, to complete the proof of this theorem, let $k \leq n/(2\log{p})$ and fix $\epsilon = 10 \mu \sqrt{2\log{p}}$. Then the claim is that $X$ satisfies $(k,\epsilon,\delta)$-StOC with $\delta \leq 4p^{-1}$. In order to establish this claim, we only need to ensure that the chosen parameters satisfy the assumptions of Lemma~\ref{lem:stoc1} and Lemma~\ref{lem:stoc2}. In this regard, note that (i) $\epsilon < 1$ because of \eqref{eqn:cp1}, and (ii) $\sqrt{k}\,\nu \leq \frac{\epsilon}{9}$ because of the assumption that $k \leq n/(2\log{p})$ and \eqref{eqn:cp2}. Therefore, since the assumption $p \geq 128$ together with $k \leq n/(2\log{p})$ implies that $16(2+ a^{-1})^2 < 72$, we obtain $\exp\left(-\frac{(\epsilon-\sqrt{k}\,\nu)^2}{16(2+ a^{-1})^2\mu^2}\right) \leq p^{-2}$. We can now combine this fact with the previously established facts to see that the threshold $\lambda = \max\Big\{\frac{1}{t}10 \mu \sqrt{n \cdot \SNR}, \frac{1}{1-t} \sqrt{2}\Big\}\sqrt{2\sigma^2\log{p}}$ guarantees that $\Pr(\whcS \not= \cS) \leq 6p^{-1}$ as long as $n \geq 2k\log{p}$ and $\beta_{\min} > 2\lambda$. Finally, note that
\begin{align}
\nonumber
	\beta_{\min} > \frac{1}{1-t}4\sqrt{\sigma^2\log{p}} \quad &\Longleftrightarrow \quad n > \frac{8(1-t)^{-2}}{\SNR_{\min}}2k\log{p}
\intertext{and}
\nonumber
	\beta_{\min} > \frac{1}{t}20 \mu \sqrt{2n\sigma^2\log{p} \cdot \SNR} \quad &\Longleftrightarrow \quad n > \bigg(\frac{c_2t^{-2}}{\MAR}2k\log{p}\bigg)^{\gamma/2}.
\end{align}
This completes the proof of the theorem.
\end{proof}

\subsubsection{Proof of Theorem~\ref{thm:pms_cp2}}
We begin by making use of the notation developed earlier in this section and conditioning on the event $\cG$ defined in \eqref{eqn:thm_ems_G} with $\epsilon = 10 \mu \sqrt{2\log{p}}$. Then it is easy to see from the proof of Theorem~\ref{thm:ems_cp} that the estimate $\whcS$ is a subset of $\cS$ because of the fact that $\|f_{\Pi^c}\|_\infty \leq \lambda$.

Next, assume without loss of generality that $z_i \equiv \beta_{(i)}$ and note from \eqref{eqn:thm_ems_S} that $|f_{\pi_i}|  \geq |\beta_{(i)}| - \lambda$ for any $i \in \{1,\dots,k\}$. Then, since $\pi_i \in \whcS$ if and only if $|f_{\pi_i}| > \lambda$, we have that $\beta_{(i)} > 2\lambda \ \Rightarrow \ \pi_i \in \whcS$. Now define $M$ to be the largest integer for which $\beta_{(M)} > 2\lambda$ holds and note that $\beta_{(M)} > 2\lambda \ \Rightarrow \ \beta_{(i)} > 2\lambda \ \Rightarrow \ \pi_i \in \whcS$ for every $i \in \{1,\dots,M\}$, which in turn implies that $\big|\cS - \whcS\big| \leq (k - M)$. Finally, note that
\begin{align}
\nonumber
	\beta_{(M)} > \frac{1}{1-t}4\sqrt{\sigma^2\log{p}} \quad &\Longleftrightarrow \quad \LAR_{M} > 8(1-t)^{-2} \left(\frac{2k\log{p}}{n \cdot \SNR}\right)
\intertext{and}
\nonumber
	\beta_{(M)} > \frac{1}{t}20 \mu \sqrt{2n\sigma^2\log{p} \cdot \SNR} \quad &\Longleftrightarrow \quad \LAR_{M} > 400t^{-2} \left(\frac{2k\log{p}}{\mu^{-2}}\right)
\end{align}
This completes the proof of the theorem since the event $\cG$ holds with probability exceeding $1 - 6p^{-1}$.\hfill\QED

\subsubsection{Proof of Theorem~\ref{thm:sig_rec}}
The first key result that we will need to prove Theorem~\ref{thm:sig_rec} is regarding the expected spectral norm of a random principal-submatrix of $(X^\tH X - I)$. The following result is mainly due to Tropp \cite{tropp:cras08} and it was first presented in the following form by Cand\`{e}s and Plan in \cite{candes:annstat09}.
\begin{proposition}[\!\!\cite{tropp:cras08,candes:annstat09}]\label{prop:ex_norm_sm}
Let $\bar{\Pi} = (\pi_1,\dots,\pi_p)$ be a random permutation of $\nN{p}$ and define $\Pi \doteq (\pi_1,\dots,\pi_k)$ for any $k \in \nN{p}$. Then, for $q = 2\log{p}$, we have
\begin{align}
 \Big(\E\Big[\big\|X^\tH_\Pi X_\Pi - I\big\|_2^q\Big]\Big)^{1/q} \leq 2^{1/q} \bigg(30\mu\log{p} + 13\sqrt{\frac{2k\|X\|_2^2\log{p}}{p}}\,\bigg)
\end{align}
provided that $k \leq p/4\|X\|_2^2$. Here, the expectation is with respect to the random permutation $\bar{\Pi}$.
\end{proposition}

Using this result, it is easy to obtain a probabilistic bound (with respect to the random permutation $\bar{\Pi}$) on the minimum and maximum singular values of a random submatrix of $X$ since, by Markov's inequality, we have that $\Pr\Big(\big\|X_\Pi^\tH X_\Pi - I\big\|_2 \geq \varsigma\Big) \leq \varsigma^{-q} \E\Big[\big\|X_\Pi^\tH X_\Pi - I\big\|_2^q\Big]$. The following result is simply a generalization of the corresponding result presented in \cite{candes:annstat09}.
\begin{proposition}[Extreme Singular Values of a Random Submatrix]\label{prop:norm_sm}
Let $\Pi = (\pi_1,\dots,\pi_k)$ denote the first $k$ elements of a random permutation of $\nN{p}$ and suppose that $\mu(X) \leq (c_1^\prime\log{p})^{-1}$ and $k \leq p / ({c_2^\prime}^2 \|X\|_2^2\log{p})$ for numerical constants $c_1^\prime \doteq 60\mathrm{e}$ and $c_2^\prime \doteq 37\mathrm{e}$. Then we have that
\begin{align}
 \Pr\Big(\big\|X_\Pi^\tH X_\Pi - I\big\|_2 \geq \mathrm{e}^{-1/2}\Big) \leq 2 p^{-1}\,.
\end{align}
\end{proposition}
\noindent Note that Proposition~\ref{prop:norm_sm} guarantees that, under certain conditions on $\mu(X)$ and $k$, every singular value of \emph{most} $n \times k$ submatrices of $X$ lies within $\big(\sqrt{1-\mathrm{e}^{-1/2}}, \sqrt{1+\mathrm{e}^{-1/2}}\,\big)$. We are now ready to provide a proof of Theorem~\ref{thm:sig_rec} that relies on this key result as well as on Lemma~\ref{lem:stoc1} and Lemma~\ref{lem:stoc2}.
\begin{proof}[Proof (Theorem~\ref{thm:sig_rec})]
The proof of this theorem follows along somewhat similar lines as the proof of Theorem~\ref{thm:ems_cp}. Specifically, by making use of the notation developed at the start of this section, we write $f = X^\tH y \equiv X^\tH X_\Pi z$ and first argue that the set of indices $\cI = \left\{i \in \nN{p} : |f_i| > \lambda\right\}$ is the same as the true model $\cS$ with high probability. Then we make use of the union bound and argue using Proposition~\ref{prop:norm_sm} that $\widehat{\beta} = \beta$ with high probability.

In this regard, recall that it was established in the proof of Theorem~\ref{thm:ems_cp} using Lemma~\ref{lem:stoc1} and Lemma~\ref{lem:stoc2} that if $X$ obeys the coherence property then it satisfies $(k,\epsilon,\delta)$-StOC with $\epsilon = 10 \mu \sqrt{2\log{p}}$ and $\delta \leq 4p^{-1}$ as long as $k \leq n/(2\log{p})$. This fact therefore implies that, under the assumptions of the theorem,\footnote{Note that the assumptions of the theorem trivially guarantee the condition $k \leq n/(2\log{p})$ since we have that $\|X\|_2^2 \geq p/n$ from elementary linear algebra.} the following inequalities hold with probability exceeding $1 - 4p^{-1}$:
\begin{align}
\label{eqn:thm_sigrec_1}
 \|f_{\Pi^c}\|_\infty &= \|X_{\Pi^c}^\tH X_\Pi z\|_\infty \leq \epsilon \|z\|_2, \quad \text{and}\\
\label{eqn:thm_sigrec_2}
 \min_{i \in \{1,\dots,k\}}|f_{\pi_i}| &\geq \beta_{\min} - \|(X_\Pi^\tH X_\Pi - I)z\|_\infty \geq \beta_{\min} - \epsilon \|z\|_2.
\end{align}
Also note that, conditioned on the probability event $\cE \doteq \big\{\big\|X_\Pi^\tH X_\Pi - I\big\|_2 < \mathrm{e}^{-1/2}\big\}$, we can write
\begin{align}
\label{eqn:thm_sigrec_rip}
 \sqrt{1-\mathrm{e}^{-1/2}}\,\|z\|_2 < \|\underbrace{X_\Pi z}_{\equiv \, y}\|_2 < \sqrt{1+\mathrm{e}^{-1/2}}\,\|z\|_2.
\end{align}
Therefore if we condition on the event $\cE$ then it trivially follows from the assumptions of the theorem and \eqref{eqn:thm_sigrec_1} and \eqref{eqn:thm_sigrec_2} that $\cI = \cS$ with probability exceeding $1 - 4p^{-1}$ since (i) $\cI \subset \cS$ because $\|f_{\Pi^c}\|_\infty < \frac{\epsilon\|y\|_2}{\sqrt{1-\mathrm{e}^{-1/2}}} \equiv \lambda$ (cf.~\eqref{eqn:thm_sigrec_1}, \eqref{eqn:thm_sigrec_rip}), and (ii) $\cI \supset \cS$ because $k \leq \mu^{-2} \MAR/(c_4^2\log{p})$ implies that $\beta_{\min} - \epsilon \|z\|_2 > \lambda \ \Rightarrow \ \min_{i \in \{1,\dots,k\}}|f_{\pi_i}| > \lambda$ (cf.~\eqref{eqn:thm_sigrec_2}, \eqref{eqn:thm_sigrec_rip}). Consequently, we conclude that $(X_{\cI})^\dag = (X_{\Pi}^\tH X_{\Pi})^{-1} X_{\Pi}^\tH$ with high probability when conditioned on the probability event $\cE$, which in turn implies that $\widehat{\beta}_{\cI} = (X_{\cI})^\dag X_\Pi z \equiv \beta_{\cS}$ with probability exceeding $1 - 4p^{-1}$ when conditioned on $\cE$. The claim of the theorem now follows trivially from the union bound and the fact that $\Pr(\cE^c) \leq 2p^{-1}$ from Proposition~\ref{prop:norm_sm} since $X$ satisfies the strong coherence property and $k \leq p/(c_3^2 \|X\|_2^2 \log{p})$.
\end{proof}

\section{Conclusions}\label{sec:disc}
In the modern statistics and signal processing literature, the lasso has arguably become the standard tool for model selection because of its computational tractability \cite{tibshirani:jrss96} and some recent theoretical guarantees \cite{meinshausen:annstat06,zhao:jmlr06,wainwright:tit09,candes:annstat09}. Nevertheless, it is desirable to study alternative solutions to the lasso since (i) it is still computationally expensive for massively large-scale inference problems (think of $p$ in the millions), (ii) it lacks theoretical guarantees beyond $k \succsim \mu^{-1}$ for the case of generic design matrices and arbitrary nonzero entries, and (iii) it requires the submatrices of the design matrix to have full rank, which seems reasonable for signal reconstruction but appears too restrictive for model selection.

In this paper, we have revisited two variants of the oft-forgotten but extremely fast one-step thresholding (OST) algorithm for model selection. One of the key insights offered by the paper in this regard is that polynomial-time model selection can be carried out even when signal reconstruction (and thereby the lasso) fails. In addition, we have established in the paper that if the $n \times p$ design matrix $X$ satisfies $\mu(X) \asymp n^{-1/2}$ and $\nu(X) \precsim n^{-1}$ then OST can perform near-optimally for the case when either (i) the minimum-to-average ratio (\MAR) of the signal is not too small or (ii) the signal-to-noise ratio (\SNR) in the measurement system is not too high. It is worth pointing out here that some researchers in the past have observed that the sorted variant of the OST (SOST) algorithm at times performs similar to or better than the lasso (see Fig.~\ref{fig:SOST_vs_Lasso} for an illustration of this in the case of an Alltop Gabor frame in $\C^{127}$). One of our main contributions in this regard is that we have taken the mystery out of this observation and explicitly specified in the paper the four key parameters of the model-selection problem, namely, $\mu(X), \nu(X), \MAR$, and $\SNR$, that determine the non-asymptotic performance of the SOST algorithm for generic (random or deterministic) design matrices and data vectors having generic (random or deterministic) nonzero entries; also, see \cite{genovese:sub09} for a comparison of our results with corresponding results recently reported in the literature.
\begin{figure*}[t]
\centering%
\subfigure[]{\includegraphics[scale=0.38]{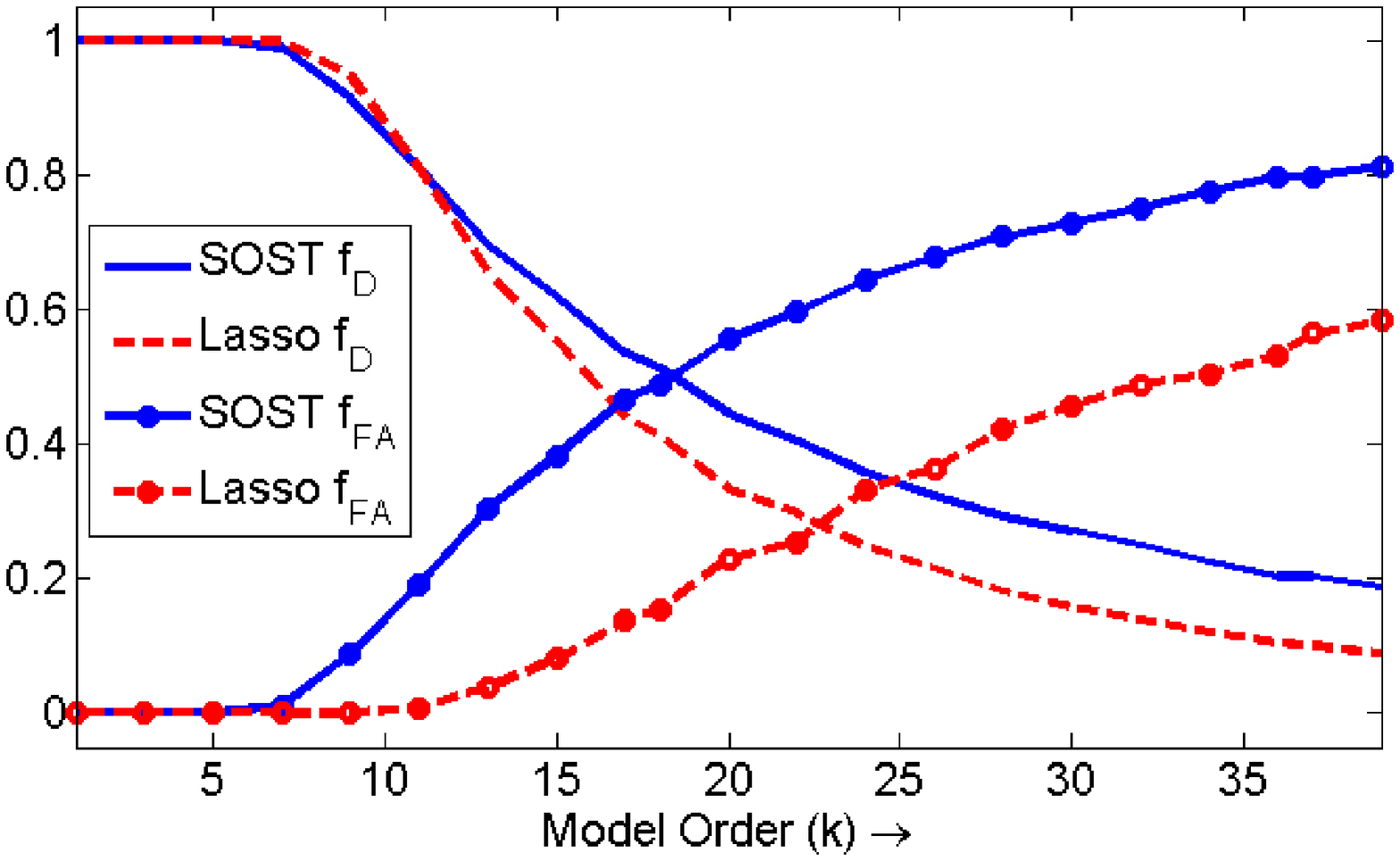}}%
\hfill%
\subfigure[]{\includegraphics[scale=0.38]{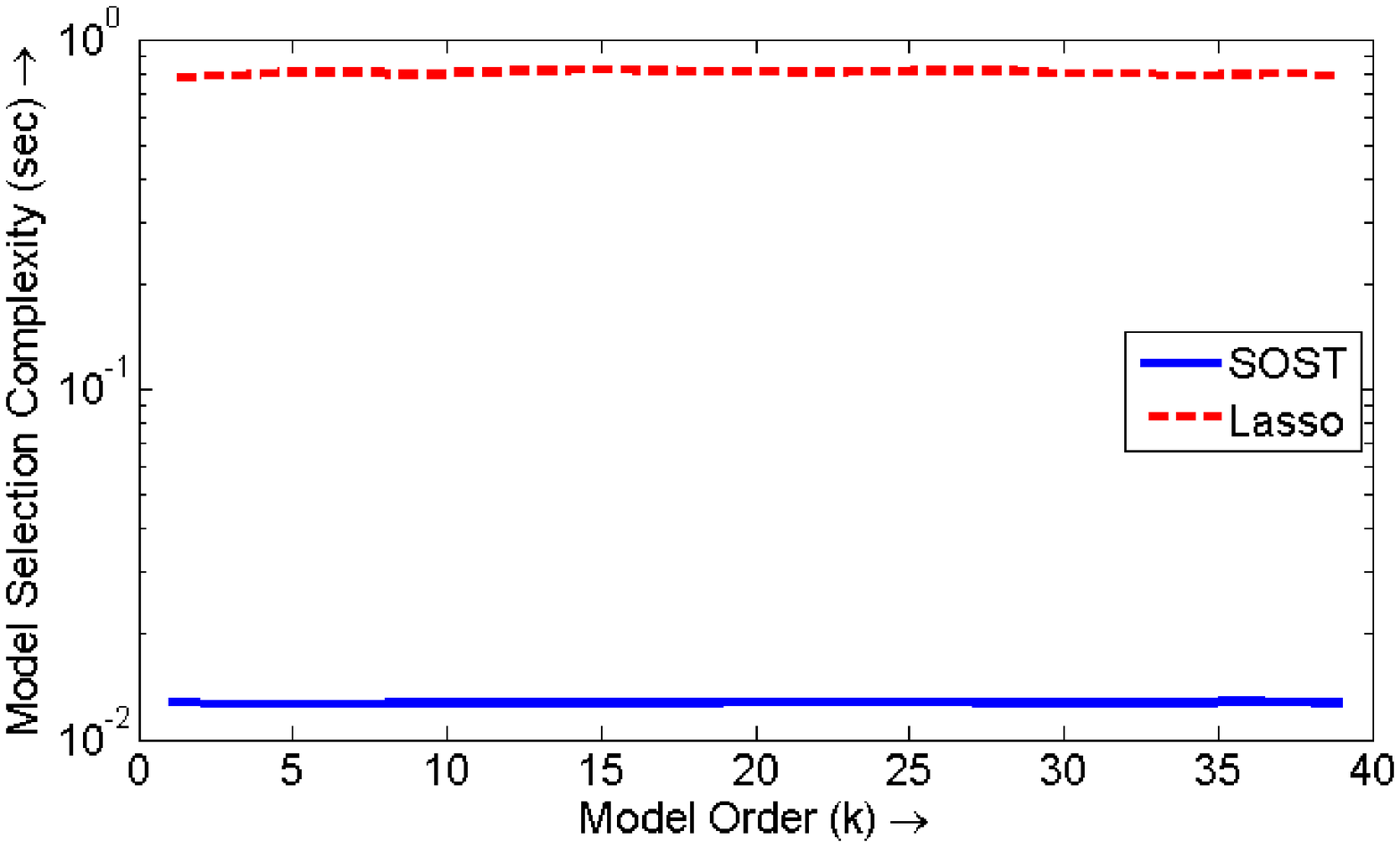}}%
\caption{Numerical comparisons between the performance of the SOST algorithm (Algorithm~\ref{alg:SOST}) and the lasso \cite{tibshirani:jrss96} using an Alltop Gabor frame. The $n \times p$ design matrix $X$ has dimensions $n = 127$ and $p = n^2$, the $\MAR$ of the signals is $1$, the $\SNR$ in the measurement system is $10 \text{ dB}$, and the noise variance is $\sigma^2 = 10^{-2}$. The matrix--vector multiplications are carried out using the fast Fourier transform, while the lasso is solved using the SpaRSA package \cite{wright:tsp09} with the regularization parameter set to $\tau = 2\sqrt{2\sigma^2\log{p}}$ \cite{candes:annstat09}. (a) Plots of the \emph{fraction of detections}, defined as $f_D = \frac{|\cS \cap \whcS|}{k}$, and the \emph{fraction of false alarms}, defined as $f_{F\!A} \defn \frac{|\whcS| - |\cS \cap \whcS|}{|\whcS|}$, versus the model order (averaged over $200$ independent trials) for both SOST and the lasso. (b) Plots of the amount of time (averaged over $200$ independent trials) that it takes SOST and the lasso to solve one model-selection problem versus the model order.}%
\label{fig:SOST_vs_Lasso}%
\end{figure*}

The second main contribution of this paper---which completely sets it apart from existing work on thresholding for model selection---is that we have proposed and analyzed a model-order agnostic threshold for the OST algorithm. The significance of this aspect of the paper can be best understood by realizing that in real-world applications it is often easier to estimate the $\SNR$ and the noise variance in the system than to estimate the true model order. In particular, we have established in the paper that the threshold $\lambda = \max\Big\{\frac{1}{t}10 \mu \sqrt{n \cdot \SNR}, \frac{1}{1-t} \sqrt{2}\Big\}\sqrt{2\sigma^2\log{p}}$ for $t \in (0,1)$ enables the OST algorithm to carry out near-optimal partial model selection. It is worth pointing out here that this threshold is rather conservative in nature for small-scale problems (see \eqref{eqnthm:pms_2}) and we believe that there is still a lot of room for improvement as far as reducing (or eliminating) some of the constants in the threshold is concerned. In particular, it is easy to see from the proof of Theorem~\ref{thm:ems_cp} that the constant $10$ in the threshold is mainly there due to a number of loose upperbounds; in fact, this constant was $24$ in a conference version of this paper \cite{bajwa:isit10} and we believe that it can be reduced even further. Some of the numerical experiments that we have carried out in this regard also seem to lend credence to our belief. Specifically, Fig.~\ref{fig:OST_Gabor} reports the results of one such experiment concerning partial model-selection performance of the OST algorithm in terms of the metrics of \emph{fraction of detections}, $f_D \defn \frac{|\cS \cap \whcS|}{k}$, and \emph{fraction of false alarms}, $f_{F\!A} \defn \frac{|\whcS| - |\cS \cap \whcS|}{|\whcS|}$, averaged over $200$ independent trials. In this experiment, the $n \times p$ design matrix $X$ corresponds to an Alltop Gabor frame in $\C^{997}$, the noise variance is $\sigma^2 = 10^{-2}$, the $\MAR$ and the $\SNR$ are chosen to be $1$ and $3 \text{ dB}$, respectively, and the initial threshold is set at $\lambda_s \defn \max\Big\{\frac{1}{t}c^\prime \mu \sqrt{n \cdot \SNR}, \frac{1}{1-t} \sqrt{2}\Big\}\sqrt{2\sigma^2\log{p}}$ with $t = (\sqrt{2} - 1)/\sqrt{2}$ and $c^\prime = 2t$. It can be easily seen from Fig.~\ref{fig:OST_Gabor} that OST successfully carries out partial model selection $(f_{F\!A} \equiv 0)$ even when the threshold is set at $0.6 \lambda_s$, which proves the somewhat conservative nature of the proposed threshold in terms of the constants.
\begin{figure*}[t]
\centering%
\includegraphics[scale=0.55]{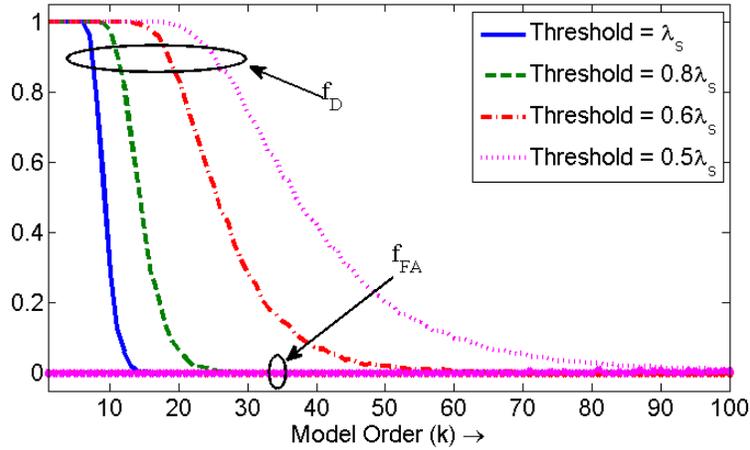}
\caption{Partial model-selection performance of the OST algorithm (averaged over $200$ independent trials) corresponding to an Alltop Gabor frame in $\C^{997}$. The $\MAR$ of the signals in this experiment is $1$, the $\SNR$ in the measurement system is $3 \text{ dB}$, and the noise variance is $\sigma^2 = 10^{-2}$.}
\label{fig:OST_Gabor}%
\end{figure*}

Finally, the third main contribution of this paper is that we have extended our results on model selection using OST to low-complexity recovery of sparse signals. In particular, within the area of low-complexity algorithms for sparse-signal recovery (such as, matching pursuit \cite{mallat:tsp93}, subspace pursuit \cite{dai:tit09}, CoSaMP \cite{needell:acha09}, and iterative hard thresholding \cite{blumensath:acha09}), we have for the first time specified polynomial-time verifiable sufficient conditions under which recovery of sparse signals having generic (random or deterministic) nonzero entries succeeds using generic (random or deterministic) design matrices. In addition, we have also provided a bound in the paper on the average coherence of generic Gabor frames and used this result to establish that an Alltop Gabor frame in $\C^n$ can be used together with the OST algorithm to successfully carry out model selection and recovery of sparse signals irrespective of the phases of the nonzero entries even if the number of nonzero entries scales almost linearly with $n$.

\appendices
\section{Concentration Inequalities}\label{app:conc_ineq}
In this appendix, we collect the various concentration inequalities that are used throughout the paper.
\begin{proposition}[The Azuma Inequality \cite{azuma:tmj67}]\label{prop:az_ineq}
Let $(\Omega,\cF,\bbP)$ be a probability space and let $(M_0,M_1,\dots,M_n)$ be a bounded difference, (real-valued) martingale sequence on $(\Omega,\cF,\bbP)$. That is, $\E[M_i] = M_{i-1}$ and $|M_i - M_{i-1}| \leq b_i$ for every $i=1,\dots,n$. Then for every $\epsilon \geq 0$, we have
\begin{align}
	\Pr\left(|M_n - M_0| \geq \epsilon\right) \leq 2\exp\left(-\frac{\epsilon^2}{2\sum\limits_{i=1}^{n}b_i^2}\right).
\end{align}
\end{proposition}
\begin{proposition}[Inner Product of Independent Gaussian Random Vectors \cite{haupt:tit08sub}]\label{prop:haupt_ip}
Let $\mathrm{x}, \mathrm{y} \in \R^n$ be two random vectors that are independently drawn from $\cN(\bzero, \sigma^2 I)$ distribution. Then for every $\epsilon \geq 0$, we have
\begin{align}
 \Pr\Big(\big|\langle\mathrm{x}, \mathrm{y}\rangle\big| \geq \epsilon\Big) \leq 2 \exp\left(-\frac{\epsilon^2}{4\sigma^2(n\sigma^2 + \epsilon/2)}\right).
\end{align}
\end{proposition}

Since we are mainly concerned with complex-valued random variables in this paper, it is helpful to state a complex version of the Azuma inequality. The following lemma is an easy consequence of Proposition~\ref{prop:az_ineq}.
\begin{lemma}[The Complex Azuma Inequality]\label{lem:az_ineq}
Let $(\Omega,\cF,\bbP)$ be a probability space and let $(M_0,M_1,\dots,M_n)$ be a bounded difference, complex-valued martingale sequence on $(\Omega,\cF,\bbP)$. That is, $\E[M_i] = M_{i-1} \in \C$ and further $|M_i - M_{i-1}| \leq b_i$ for every $i=1,\dots,n$. Then for every $\epsilon \geq 0$, we have
\begin{align}
	\Pr\left(|M_n - M_0| \geq \epsilon\right) \leq 4\exp\left(-\frac{\epsilon^2}{4\sum\limits_{i=1}^{n}b_i^2}\right).
\end{align}
\end{lemma}
\begin{proof}
To establish this lemma, first define $S_i \doteq \text{Re}(M_i)$ and $T_i \doteq \text{Im}(M_i)$. Further, notice that since $\E[M_i] = M_{i-1}$ and $|M_i - M_{i-1}| \leq b_i$, we equivalently have that: (i) $\E[S_i] = S_{i-1}$ and $|S_i - S_{i-1}| \leq b_i$, and (ii) $\E[T_i] = T_{i-1}$ and $|T_i - T_{i-1}| \leq b_i$. Therefore, we have that $(S_0,S_1,\dots,S_n)$ and $(T_0,T_1,\dots,T_n)$ are bounded difference, real-valued martingale sequences on $(\Omega,\cF,\bbP)$ and hence
\begin{align}
	\Pr\left(|M_n - M_0| \geq \epsilon\right) \stackrel{(a)}{\leq} \Pr\left(|S_n - S_0| \geq \frac{\epsilon}{\sqrt{2}}\right) + \Pr\left(|T_n - T_0| \geq \frac{\epsilon}{\sqrt{2}}\right) \stackrel{(b)}{\leq} 4\exp\left(-\frac{\epsilon^2}{4\sum\limits_{i=1}^{n}b_i^2}\right)
\end{align}
where $(a)$ follows from a simple union bounding argument and $(b)$ follows from the Azuma inequality.
\end{proof}
\begin{lemma}[$\ell_\infty$-Norm of the Projection of a Complex Gaussian Vector]\label{lem:gaussian}
Let $X$ be a real- or complex-valued $n \times p$ matrix having unit $\ell_2$-norm columns and let $\eta$ be a $p \times 1$ vector having entries independently distributed as $\cCN(0,\sigma^2)$. Then for any $\epsilon > 0$, we have
\begin{align}
    \Pr\left(\|X^\tH\eta\|_\infty \geq \sigma \epsilon\right) <
        \frac{4 p}{\sqrt{2\pi}} \cdot \frac{\exp(-\epsilon^2/2)}{\epsilon} \, .
\end{align}
\end{lemma}
\begin{proof}
Assume without loss of generality that $\sigma = 1$, since the general case follows from a simple rescaling argument. Let $\mathrm{x}_1, \dots, \mathrm{x}_p \in \C^n$ be the $p$ columns of $X$ and define
\begin{align}
    z_i \doteq \mathrm{x}_i^\tH \eta, \ i=1,\dots,p.
\end{align}
Note that the $z_i$'s are identically (but not independently) distributed as $z_i \sim \cCN(0,1)$, which follows from the fact that $\eta_i \stackrel{i.i.d.}{\sim} \cCN(0,1)$ and the columns of $X$ have unit $\ell_2$-norms. The rest of the proof is pretty elementary and follows from the facts that
\begin{align}
    \nonumber
    \Pr\left(\|X^\tH\eta\|_\infty \geq \epsilon\right) &\stackrel{(a)}{\leq} p \cdot \Pr\left(|\text{Re}(z_1)|^2 + |\text{Im}(z_1)|^2 \geq \epsilon^2\right)\\
    \nonumber
        &\stackrel{(b)}{\leq} 2p \cdot \Pr\left(|\text{Re}(z_1)| \geq \frac{\epsilon}{\sqrt{2}}\right) = 2p \cdot 2 Q(\epsilon)\\
        &\stackrel{(c)}{<} \frac{4p}{\sqrt{2\pi}} \cdot \frac{\exp(-\epsilon^2/2)}{\epsilon} \, .
\end{align}
Here, $(a)$ follows by taking a union bound over the event $\bigcup_i\{|z_i| \geq \epsilon\}$, $(b)$ follows from taking a union bound over the event $\{|\text{Re}(z_1)| \geq \epsilon/\sqrt{2}\} \cup \{|\text{Im}(z_1)| \geq \epsilon/\sqrt{2}\}$ and noting that the real and imaginary parts of $z_i$'s are identically distributed as $\cN(0,\frac{1}{2})$, and $(c)$ follows by upper bounding the \emph{complementary cumulative distribution function} as $Q(\epsilon) < \frac{1}{\sqrt{2\pi} \epsilon} \exp(-\frac{1}{2}\epsilon^2)$ \cite{kay:98b}.
\end{proof}


\end{document}